\documentclass[12pt,a4paper]{amsart}
\usepackage{amsmath, amsthm, amscd, amsfonts,graphicx}

\usepackage{color}
\usepackage{enumerate}
\numberwithin{equation}{section}
\theoremstyle{plain}
 \newtheorem{theorem}{Theorem}[section]
 \newtheorem{lemma}[theorem]{Lemma}
 \newtheorem{proposition}[theorem]{Proposition}
 
 \newtheorem{corollary}[theorem]{Corollary}
\theoremstyle{definition}

\newenvironment{enumeratei}{\begin{enumerate}[\quad\upshape (i)]} {\end{enumerate}}

\newcommand \dual [1] {{#1^{\textup{dual}}}}
\newcommand \tdual [1] {{#1{}\kern-1.3pt^{\textup{dual}}}}
\newcommand \pow [1] {\textup{Pow}(#1)}
\newcommand \printv {\textup{PrIntv}}
\newcommand \Lat [1] {\mathcal L_{\textup{lat}}(#1)}
\newcommand \Cgeo [1] {\mathcal G_{\textup{conv}}(#1)}
\newcommand \hsz {{\triangle_{A_0,A_1,A_2}}}
\newcommand \khsz[1] {{\triangle_{A_0,A_1,A_2}^{(#1)}}}
\newcommand \ngbhood [1] {\hal N_{\kern-2pt #1}}
\newcommand \Con {\textup{Con}}
\newcommand \con {\textup{con}}
\newcommand \targref [1] {\textup{(T\ref{#1})}}
\newcommand \lmskip {\hskip 6pt} %See Lemma~\ref{lemmaTechnmwsgT}
\newcommand \quot [1] {``#1''}
\newcommand \hal[1]{{\mathcal #1\kern 0.1pt}}
\renewcommand \phi {\varphi}
\renewcommand \rho {\varrho}
\newcommand \thresh [1] {m^{\kern-0.5pt(#1)}}
\newcommand \abund [2] {\textup{Abd}(#1 , #2)}
\newcommand \enlarge [2] {\textup{OpExt}(#1, #2)}
\newcommand \pbxskip {\hskip 2 pt}
\newcommand \nplu {\mathbb N^+}
\newcommand \dir[1]{\textup{dir}(#1)}
\newcommand \Cunit {\hal C_{\textup{unit}}}

\newcommand \inter[1] {\textup{Int}(#1)}

\renewcommand \epsilon {{\boldsymbol\varepsilon}}
\newcommand \trap {\textup{Trp}}
\newcommand \goodset {H}

\newcommand \dist[2]{\textup{dist}(#1,#2)}

\newcommand \eeqref[1] { \overset{\textup{\eqref{#1}}}=  }
\newcommand \ssty[1]{\scriptscriptstyle{#1}}

\newcommand \hthet[2] {\boldsymbol\chi_{#1,#2}}
\newcommand \loosin {\overset{\ssty{\textup{loose}}}\subset}
\newcommand \comet [2]{{\textup{Comet}(#1,#2)}}
\newcommand \rhullo [1] {\textup{Conv}_{\real^{#1}}} % Real HULL Operator in  R^#1
\newcommand\set [1]{\{#1\}}
\newcommand \tuple [1] {\langle #1 \rangle}
\newcommand \pvpair [2] {\tuple{#1;#2}}
\newcommand \pair [2] {\tuple{#1,#2}}
\newcommand \real {\mathbb R}
\newcommand \preal {\real^{2}}
\newcommand \nothing [1] {}

\newcommand \red [1] {{\color{red}#1\color{black}}}

\newcommand \tbf[1] {\textbf{#1}}  
\begin{document}
\title[A combinatorial property of compact sets]
{A convex combinatorial property of compact sets in the plane and its roots in lattice theory}

\author[G.\ Cz\'edli]{G\'abor Cz\'edli}
\address{G\'abor Cz\'edli, University of Szeged, Bolyai Institute, Szeged,
Aradi v\'ertan\'uk tere 1, Hungary 6720}
\email{czedli@math.u-szeged.hu}
\urladdr{http://www.math.u-szeged.hu/\textasciitilde{}czedli/}

\author[\'A.\ Kurusa]{\'Arp\'ad Kurusa}
\address{\'Arp\'ad Kurusa, University of Szeged, Bolyai Institute\\Szeged, Aradi v\'ertan\'uk tere 1, Hungary 6720}
\email{kurusa@math.u-szeged.hu}
\urladdr{http://www.math.u-szeged.hu/\textasciitilde{}kurusa/}

%\thanks{This research of the first author and that of the second author was supported by the Hungarian Research Grants KH 126581 and K 116451, respectively.}

%52C (1991-now) Discrete geometry 
%52C99 (1991-now) None of the above, but in this section 
%52A01 (1980-now) Axiomatic and generalized convexity


\begin{abstract} 
K.\ Adaricheva and M.\ Bolat have recently proved that if $\,\hal U_0$ and $\,\hal U_1$ are circles in a triangle with vertices $A_0,A_1,A_2$, then there exist $j\in \set{0,1,2}$ and $k\in\set{0,1}$ such that $\,\hal U_{1-k}$ is included in the convex hull of $\,\hal U_k\cup(\set{A_0,A_1, A_2}\setminus\set{A_j})$. 
One could say disks instead of circles.
Here we prove the existence of such a $j$ and $k$ for the more general case where $\,\hal U_0$ and $\,\hal U_1$ are compact  sets in the plane such that $\,\hal U_1$ is obtained from $\,\hal U_0$ by a positive homothety or by a translation. 
Also, we give a short survey to show how lattice theoretical antecedents, including a series of papers on planar semimodular lattices by G.\ Gr\"atzer and E.\ Knapp, lead to our result.
\end{abstract}

\dedicatory{Dedicated to George A.~Gr\"atzer on the occasion of the fifty-fifth anniversary of the Gr\"atzer--Schmidt Theorem and the fortieth anniversary of his monograph ``General Lattice Theory"}

\date{\red{\hfill CET 03:20, \qquad  July 10, 2018}}

\maketitle

\section{Aim and outline}

\subsection*{Our goal}
Apart from a survey of historical nature, which the reader can skip over if he is interested only in the main result, this paper belongs to elementary combinatorial \emph{geometry}.  The motivation and the excuse that this paper is submitted to CGASA are the following.
The first author has recently written a short biographical paper \cite{czgGGa} to celebrate professor George A.\ Gr\"atzer, and also an interview \cite{czgGGb} with him; at the time of this writing, both have already appeared online in CGASA. Although \cite{czgGGa} mentions that G.\ Gr\"atzer's purely lattice theoretical results have lead to results in geometry, no  detail on the transition from lattice theory to geometry is given there. This paper, besides presenting a recent result in geometry, exemplifies how such a purely lattice theoretical target as studying congruence lattices of finite lattices can lead, surprisingly, to some progress in geometry.

The real plane and the usual convex hull operator on it will be denoted by $\preal$ and $\rhullo 2$. 
In order to formulate our result, we need the following two kinds of planar transformations, that is, $\preal\to\preal$ maps. 
Given $P\in\preal$ and $0<\lambda\in\real$,  the \emph{positive homothety} with (homothetic) center $P$ and ratio $\lambda$ is defined by  
\begin{equation}
\text{
$\hthet P\lambda \colon \preal\to \preal$ by  $X \mapsto (1-\lambda)P+\lambda X = P+\lambda(X-P)$.
}
\label{eqtxthmThdF}
\end{equation}
The more general concept of homotheties where $\lambda$ can also be negative is not needed in the present paper. 
For a given $P\in \preal$, the map $\preal\to\preal$, defined by $X\mapsto P+X$, is a \emph{translation}. 
Our main goal is to prove the following theorem.

\begin{theorem}\label{thmmain}
Let $A_0,A_1,A_2\in\preal$ be points of the plane $\preal$. 
Also, let $\,\hal U_0\subset \preal$ and $\,\hal U_1\subset \preal$ be compact sets such that at least one of the following three conditions holds:
\begin{enumerate}[\upshape{\indent}(a)]
\item\label{thmmaina} $\,\hal U_1$ is a positive homothetic image of $\,\hal U_0$, that is, $\,\hal U_1=\hthet P\lambda(\hal U_0)$ for some $P\in\preal$ and $0<\lambda\in\real$;
\item\label{thmmainb} $\,\hal U_1$ is obtained from $\,\hal U_0$ by a translation;
\item\label{thmmainc} at least one of $\,\hal U_0$ and $\,\hal U_1$ is a singleton.
\end{enumerate}
With these assumptions, 
\begin{equation}
\left\{\pbxskip
\parbox{9.5cm}
{if 
$\,\,\hal U_0\cup \,\hal U_1\subseteq \rhullo 2(\set{A_0,A_1,A_2})$,  
then there exist subscripts $j\in \set{0,1,2}$ and $k\in\set{0,1}$ such that\\ 
\leftline{\hfil\hfil$
\hal U_{1-k} \subseteq \rhullo 2\bigl( \hal U_k\cup (\set{A_0,A_1, 
A_2}\setminus\set{A_j})\bigr).
$\hfil}
}
\right.
\label{eqpbxThmlnygrSmzWrB}
\end{equation}
\end{theorem}

Since at least one of the conditions \ref{thmmain}\eqref{thmmaina}, \ref{thmmain}\eqref{thmmainb}, and \ref{thmmain}\eqref{thmmainc} holds for any two circles, the following result of Adaricheva and Bolat becomes an immediate consequence of  Theorem~\ref{thmmain}.

\begin{corollary}[{Adaricheva and Bolat~\cite[Theorem 3.1]{kabolat}}]\label{corolmain}
Let $\,\hal U_0$ and $\,\hal U_1$ be \emph{circles} in the plane $\preal$. Then  \eqref{eqpbxThmlnygrSmzWrB} holds for all $A_0,A_1,A_2\in\preal$.
\end{corollary}

Four comments are appropriate here. First,
according to another terminology, 
 positive homotheties (in our sense) and translations generate the group of positive  \emph{homothety-translations}, and one might think of using homothety-translations in Theorem~\ref{thmmain}. It will be pointed out in Lemma~\ref{lemmafrcSPpRtT} that we would not obtain a new result in this way, because a homothety-translation is always a homothety or a translation.
Since the rest of the paper focuses mainly on homotheties in our sense, we use disjunction rather than the 
hyphened form ``homothety-translation''.
Second, it is easy to see that \eqref{eqpbxThmlnygrSmzWrB} does not hold for two \emph{arbitrary} compact  sets, so the disjunction of  \eqref{thmmaina}, \eqref{thmmainb}, and \eqref{thmmainc} cannot be omitted from Theorem~\ref{thmmain}; see also Cz\'edli~\cite{czgCharc} for related information.
Third, Example 4.1 of Cz\'edli~\cite{czganeasy} rules out the possibility of generalizing Theorem~\ref{thmmain} for higher dimensions. Fourth, one may ask whether stipulating the \emph{compactness} of  $\,\hal U_0$ and $\,\hal U_1$ is essential in Theorem~\ref{thmmain}. Clearly, ``compact'' could be replaced by ``(topologically) closed'', since a non-compact closed set cannot be a subset of the triangle $\rhullo 2(\set{A_0,A_1,A_2})$,
but this trivial rewording would not be a valuable improvement. 
The situation
\begin{align*}
A_0&:=\pair 6 0,\,\,\, A_1:=\pair{-3}{3\sqrt3},\,\,\, A_2:=\pair{-3}{-3\sqrt3},\cr
\hal U_0&:=\set{\pair x y: x^2+y^2<1} \cup \set{\pair x y: x^2+y^2=1\text{ and $x$ is rational}}, \cr
\hal U_1&:=\set{\pair x y: x^2+y^2<1} \cup \set{\pair x y: x^2+y^2=1\text{ and $x$ is irrational}}, 
\end{align*}
exemplifies that Theorem~\ref{thmmain} would fail without requiring the compactness of $\,\hal U_0$ and $\,\hal U_1$.

\subsection*{Prerequisites and outline} 
No special prerequisites are required; practically,  every mathematician with usual M.Sc.\ background can understand the proof of Theorem~\ref{thmmain}. On the other hand, Section~\ref{sect2} is of historical nature and can be interesting mainly for specialists. 

The rest of the paper is structured as follows. 
In Section~\ref{sect2}, starting from lattice theoretical results including Gr\"atzer and Knapp~\cite{gknapp1,gknapp2,gknapp25,gknapp3,gknapp4}, we survey how lattice theoretical results lead to the present paper. Also, we say a few words on some similar results that belong to combinatorial geometry. Section~\ref{sect3} proves  Theorem~\ref{thmmain}  only for the particular case where the convex closures of the compact sets $\,\hal U_0$ and $\,\hal U_1$ are ``edge-free'' (to be defined later). Section~\ref{sectiongettingrid} reduces the general case to the edge-free case and completes the proof of  Theorem~\ref{thmmain}.

\section{From George Gr\"atzer's lattice theoretical papers to  geometry}\label{sect2}
Congruence lattices of finite lattices are well known to form 
George Gr\"atzer's favorite research topic, in which he has proved many nice and deep results; see, for example, the last section in the  biographical paper \cite{czgGGa} by the first author. 
At first sight, it is not so easy to imagine interesting links between this topic and geometry. The aim of this section is to present such a link by explaining how some of Gr\"atzer's \emph{purely lattice theoretical} results have lead to the present paper and other papers in \emph{geometry}. 
Instead of over-packing this section with too many definitions and statements, we are going to focus on \emph{links} connecting results and publications.  This explains that, in this section, some definitions are given only after discussing the links related to them. 
Although a part of this exposition is based on the experience of the first author, the link between lattice theory and the geometrical topic of the present paper is hopefully more than just a personal feeling.

\subsection{Planar semimodular lattices and their congruences}\label{subsect2a}
On November 28, 2006, Gr\"atzer and his student, Edward Knapp submitted their first paper, \cite{gknapp1}, to Acta Sci.\ Math.\ (Szeged) on planar semimodular lattices. 
A lattice $\hal L=\tuple{L;\vee,\wedge}$ is \emph{semimodular} if, for every $a\in L$, the map $L\to L$, defined by $x\mapsto a\vee x$, preserves the ``covers or equal'' relation $\preceq$; as usual,
$a\preceq b$ stands for $|\set{x\in L: a\leq x\leq b}|\in\set{1,2}$.
A lattice is \emph{planar} if it is finite and has a Hasse diagram that is also a planar graph.  Their first paper, \cite{gknapp1}, were soon followed by Gr\"atzer and  Knapp  \cite{gknapp2,gknapp25,gknapp3,gknapp4}. After giving a structural description of planar semimodular lattices, they proved nice results on the congruence lattices of these lattices in \cite{gknapp2}, \cite{gknapp3}, and \cite{gknapp4}. 

The lattice $\Con(\hal L)$ of all congruence relations of  $\hal L$ is the \emph{congruence lattice} of $\hal L$, and it is known to be a distributive algebraic lattice by an old result of Funayama and Nakayama~\cite{funayamanakayama}. It is a milestone in the history of lattice theory that not every distributive lattice $\hal D$ can be represented in the form of $\Con(\hal L)$; this famous result is due to Wehrung~\cite{wehrung}. However, 
\begin{equation}
\left\{\pbxskip
\parbox{8.7cm}{every finite distributive lattice $\hal D$ can be represented, up to isomorphism, as $\Con(\hal L)$ where $\hal L$ is a finite lattice.}\right.
\label{eqpboxDlRrpThtl}
\end{equation}
 This result is due to Dilworth, see \cite{bogartdilworth}, but it was not published until Gr\"atzer and Schmidt~\cite{ggsch62}.
There are several ways of generalizing \eqref{eqpboxDlRrpThtl}; the first four of the following targets are due to G.\ Gr\"atzer or to G.\ Gr\"atzer and E.\ T.\ Schmidt.
\begin{enumerate}[\upshape\kern 4pt (T1)]   %[\upshape\kern 4pt ---]
\item\label{targa} Find an $\hal L$ with nice properties in addition to $\Con(\hal L)\cong \hal D$,
\item\label{targb} find an $\hal L$ of size being as small as possible,
\item\label{targc} represent two or even more finite distributive lattices and certain isotone maps among them simultaneously,
\item\label{targd} represent a finite ordered set (also known as a  poset) as the ordered set of \emph{principal congruences} of a finite lattice, and
\item\label{targe} combine some of the targets above.
\end{enumerate}
There are dozens of results and papers addressing these targets.
The monograph Gr\a"atzer~\cite{ggbypicture1} surveyed the results of this kind available before 2006. Ten years later, the new edition \cite{ggbypicture2} became much more extensive, and the progress has not yet finished. The series of papers by Gr\"atzer and Knapp fits well into the targets listed above. Indeed, \cite{gknapp3} fits \targref{targa} by providing a \emph{rectangular lattice} $\hal L$ while, fitting both \targref{targa} and \targref{targb}, \cite{gknapp4} minimizes the size of this rectangular $\hal L$. A \emph{rectangular lattice} is a planar semimodular lattice with a pair $\pair u v\neq\pair 0 1$ of double irreducible elements such that $u\wedge v=0$ and $u\vee v=1$; 
these lattices have nice  rectangle-shaped planar diagrams.

Next, in their 2010 paper, Gr\"atzer and Nation~\cite{gratzernation} proved a stronger form of the classical Jordan--H\"older theorem for groups from the nineteenth century. 
Here we formulate their result only for  groups, but note that both \cite{gratzernation} and  ~\cite{czgschtJH}, to be mentioned soon, formulated the results for semimodular lattices. For subnormal subgroups $A\triangleleft B$ 
and $C\triangleleft D$ of a given group $G$, the quotient  $B/A$ is said to be \emph{subnormally down-and-up projective} to $D/C$ if there are subnormal subgroups $E\triangleleft F$ such that $AF=B$, 
$A\cap F=E$, $CF=D$, and $C\cap F=E$.  Gr\"atzer and Nation's result for finite groups says that whenever $\set{1}=X_0\triangleleft X_1 \triangleleft\dots \triangleleft  X_n=G$ and  $\set{1}=Y_0\triangleleft Y_1\triangleleft\dots \triangleleft  Y_m=G$ are composition series of a group $G$, then 
$n=m$ and
\begin{equation}
\left\{\pbxskip
\parbox{10.0cm}{there exists a permutation $\pi\colon\set{1,\dots,n}\to \set{1,\dots,n}$ such that $X_i/X_{i-1}$ is subnormally down-and-up projective to $Y_{\pi(i)}/Y_{\pi(i)-1}$ for all $i\in \set{1,\dots,n}$.}\right.
\label{eqpbxJHcPsdzBmTr}
\end{equation} 
(The original Jordan--H\"older theorem states only that the quotient groups 
$X_i/X_{i-1}$ and $Y_{\pi(i)}/Y_{\pi(i)-1}$  are isomorphic, because they are in the \emph{transitive closure} of subnormal down-and-up projectivity.)
Not much later,  Cz\'edli and Schmidt~\cite{czgschtJH} added that 
$\pi$ in \eqref{eqpbxJHcPsdzBmTr} is uniquely determined. The proof in \cite{czgschtJH} is based on \emph{slim} planar semimodular lattices; this concept was introduced in G\"atzer and  Knapp  \cite{gknapp1}: a planar semimodular lattice is \emph{slim} if  $\hal M_3$, the five-element nondistributive modular lattice, cannot be embedded into it in a cover-preserving way.

Next, \targref{targa}--\targref{targe} and the applicability of slim semimodular lattices for groups motivated further results on the structure of slim semimodular lattices, including Cz\'edli~\cite{czgmtxslim},
Cz\'edli and Gr\"atzer~\cite{czgggresection},  Cz\'edli, Ozsv\'art and Udvari~\cite{czgozsudv}, and Cz\'edli and Schmidt~\cite{czgschtslim1} and \cite{czgschtslim2}. 
Some results on the congruences and congruence lattices of these lattices,
including 
Cz\'edli~\cite{czgrectangrephomo}, \cite{czgpatchext}, \cite{czgnoteconslim},
Cz\'edli and Makay~\cite{czgmakay}, 
Gr\"atzer~\cite{ggswinglemma}, \cite{ggOnCzG}, and  Gr\"atzer  and Schmidt~\cite{ggsch14} and \cite{ggschshort14} have also been proved.  Neither of these two lists is complete; see the book sections  Cz\'edli and  Gr\"atzer~\cite{czgggbooksection} and Gr\"atzer~\cite{ggBSect2014} and the monograph Gr\"atzer~\cite{ggbypicture2} for additional information and references.

\subsection{Convex geometries as combinatorial structures}\label{subsect2b} 
It was an anonymous referee of Cz\'edli, Ozsv\'art and Udvari~\cite{czgozsudv} who pointed out that slim semimodular lattices can be viewed as convex geometries of convex dimension at most 2; see Proposition~\ref{propzBjS} and the paragraph following it in the present paper.  As the first consequence of this remark, 
Adaricheva and Cz\'edli~\cite{kaczg} and Cz\'edli~\cite{czgcoord}
gave a lattice theoretical new proof of the ``coordinatizability'' of convex geometries by permutations; the original combinatorial result is due to Edelman and Jamison~\cite{edeljam}. 

As we know from  Monjardet~\cite{monjardet}, convex geometries 
are so important that they had been discovered or rediscovered in many equivalent forms even by 1985 not only as  combinatorial structures but also as lattices. This explains that the terminology is far from being unique. Here we go after the terminology used in Cz\'edli~\cite{czgcoord} even when much older results are cited. 
If the reader is interested in further information on convex geometries, he may turn to \cite{czgcoord} for a limited survey or to Adaricheva and Nation~\cite{kirajbbooksection} for a more extensive treatise.  To keep the  size limited, we do not mention antimatroids and  meet-distributivity; see the survey part of Cz\'edli~\cite{czgcoord} for references on them.

In order to give a \emph{combinatorial definition}, the \emph{power set} of a given \emph{finite} set $E$ will be denoted by 
  $\pow E:=\set{X: X\subseteq E}$. 
If a map $\Phi\colon\pow E\to\pow E$ satisfies the rules $X\subseteq \Phi(X)\subseteq \Phi(Y)=\Phi(\Phi(Y))$ for all $X\subseteq Y\subseteq E$, 
then $\Phi$ is a \emph{closure operator} over the set $E$. 
A pair  $\pvpair E \Phi$ is a \emph{convex geometry} if $E$ is a nonempty set, $\Phi$ is a closure operator over $E$,  $\Phi(\emptyset)=\emptyset$, and, for all $p,q\in E$ and $X=\Phi(X)\in \pow E$, the \emph{anti-exchange property}
\begin{equation}
(\,\,p\neq q,\,\,p\notin X,\,\, q\notin X, \,\, p\in\Phi(X\cup \set{q})\,\,)
\Rightarrow q\notin\Phi(X\cup \set{p})
\end{equation} 
holds. For example, if $E$ is a finite set of points of $\preal$, then we obtain a convex geometry $\pvpair E \Phi$ by letting
\begin{equation}
\Phi\colon \pow E\to\pow E\,\text{ defined by }\, \Phi(X):=E\cap\rhullo 2 (X).
\label{eqtxthFgjRcnpmnt}
\end{equation}

The \emph{dual} of a lattice $\hal K=\tuple{K; \vee,\wedge}$ is denoted by    $\dual{\hal K}:=\tuple{K; \wedge,\vee}$. For $x\in \hal K$, let $x^\ast:=\bigvee\set{y: x\prec y}$. Let $\hal M_3$ denote the five-element modular non-distributive lattice.  By a \emph{join-distributive lattice} we mean a semimodular lattice of finite length that does not include $\hal M_3$  a sublattice.  (This concept should not  be confused with join-semidistributivity.)  Equivalently,  a semimodular lattice $\hal K$ of finite length is \emph{join-distributive} if the interval $[x,x^\ast]$ is a distributive lattice for all $x\in \hal K\setminus\set 1$; this is the definition that explains the current terminology. From the literature, Cz\'edli~\cite[Proposition 2.1]{czgcoord}  collects eight equivalent definitions of join-distributivity; the oldest one of them is due to Dilworth~\cite{dilworth1940}.

Given a convex geometry $\pvpair E \Phi$,  the set $\set{X\in \pow E: X=\Phi(X)}$ of \emph{closed sets} forms a lattice with respect to set inclusion $\subseteq$. The \emph{dual} of this lattice will be denoted by  $\Lat{\pvpair E \Phi}$. As usual, a   set $X\in\pow E$ is called \emph{open} if $E\setminus X$ is closed. With this terminology, $\Lat{\pvpair E \Phi}$ can be considered as the lattice of open subsets of $E$  with respect to set inclusion $\subseteq$.

For a finite lattice $\hal K$, the set $J(\hal K)$ of (non-zero)  \emph{join-irreducible elements} is defined as 
$\set{x\in \hal K:\text{there is exactly one }y\in\hal K\text{ with }y\prec x}$. Next, for a finite lattice $\hal L$, we define a closure operator
\begin{align*}
\Phi_{\dual{\hal L}}&\colon \pow{J(\dual{\hal L})} \to \pow{J(\dual{\hal L})}\,\,\text{ by}\cr
\Phi_{\dual{\hal L}}(X)&:=
\Bigl\{ y\in J(\dual{\hal L}): y  \mathrel{\leq_{\dual{\hal L}}}
\bigvee_{\dual{\hal L}} X \Bigr\}, \text{ and we let }
\cr
\Cgeo{\hal L}&:=\pvpair{J(\dual{\hal L})}{\Phi_{\dual{\hal L}}}.
\end{align*}
Of course, the inequality above is equivalent to  $y\geq \bigwedge X$ in $\hal L$ and  $J(\dual{\hal L})$ equals the set of meet-irreducible elements of $\hal L$.
The following proposition, cited as the combination of Proposition 7.3 and  Lemma 7.4   in \cite{czgcoord}, is due to   Adaricheva, Gorbunov,  and Tumanov~\cite{kiragorbunov} and  Edelman~\cite{edelman}.
 
\begin{proposition}\label{propzBjS}
Let $\pvpair E\Phi$  and   $\hal L$ be a convex geometry and a join-distributive lattice, respectively. Then  
$\Lat{\pvpair E \Phi}$ is a join-distributive lattice, $\Cgeo{\hal L}$ is a convex geometry, and, in addition, we have that $\Cgeo{\Lat{\pvpair E \Phi}}\cong \pvpair E \Phi$ and 
$\Lat{\Cgeo{\hal L}}\cong \hal L$.
\end{proposition}

This proposition allows us to say that convex geometries and join-distributive lattices capture basically the same concept. Based on Proposition~\ref{propzBjS} and the theory of planar semimodular lattices summarized in Cz\'edli and Gr\"atzer~\cite{czgggbooksection}, we can say that a convex geometry $\pvpair E \Phi$ is of 
\emph{convex dimension} at most 2 if $\Lat{\pvpair E \Phi}$ is a slim semimodular lattice.

\subsection{From lattices to convex geometries  by means of trajectories} \label{subsect2c}
In order to describe the first step from Gr\"atzer and Knapp~\cite{gknapp1,gknapp2,gknapp25,gknapp3,gknapp4} and Gr\"atzer and Nation~\cite{gratzernation}
towards geometry, we need to define trajectories. If $a\prec b$ in a finite lattice $\hal K$, then 
$[a,b]$ is called a \emph{prime interval} of $\hal K$. The \emph{set of prime intervals} of $\hal K$ will be denoted by $\printv(\hal K)$. 
Two prime intervals,  $[a_0,b_0], [a_1,b_1]\in \printv(\hal K)$, are \emph{consecutive} if $a_i=a_{1-i}\wedge b_i$ and $b_{1-i}=a_{1-i}\vee b_i$ hold for some $i\in\set{0,1}$. The reflexive-transitive closure of consecutiveness is an equivalence relation on $\printv(\hal K)$, and its classes are called the \emph{trajectories} of $\hal K$.

Trajectories  were introduced in Cz\'edli and Schmidt~\cite{czgschtJH}, and they played the key role in proving the uniqueness of $\pi$ in \eqref{eqpbxJHcPsdzBmTr}. Soon afterwards, trajectories were intensively used when dealing with congruence lattices of  slim planar semimodular lattices, because for $x\prec y$ and $a\prec b$ is such a lattice, one can describe with the help of trajectories whether 
the least congruence $\con(a,b)$ collapsing $\pair a b$ contains (in other words, collapses) $\pair x y$. Later, similarly to trajectories, a beautiful description of the containment $\pair x y\in \con(a,b)$ was described by Gr\"atzer's Swing Lemma; see Gr\"atzer~\cite{ggswinglemma}, and see Cz\'edli, Gr\"atzer, and Lakser~\cite{czggglakser} and Cz\'edli and Makay~\cite{czgmakay} for a generalization and for alternative approaches. Note that 
Lemma~2.36 in Freese, Je\v zek, and Nation~\cite[page 41]{freesejezeknation}, which is due to
J\'onsson and Nation~\cite{jonssonnation} originally, offers an alternative way to describe whether $\pair x y\in \con(a,b)$.

For distinct prime intervals $[a_0,b_0], [a_1,b_1]\in \printv(\hal K)$, we say that $[a_0,b_0]$ and $[a_1,b_1]$ are \emph{comparable} if either $b_0\leq a_1$, or $b_1\leq a_0$. It was proved in Adaricheva and Cz\'edli~\cite{kaczg} that 
\begin{equation}
\left\{\pbxskip
\parbox{9.4cm}{a finite semimodular lattice $\hal L$ is join-distributive if and only if no two distinct comparable prime intervals of $\hal L$ belong to the same trajectory.}\right.
\label{eqpbxTrhGnhrsn}
\end{equation}
Combining \eqref{eqpbxTrhGnhrsn} with Proposition~\ref{propzBjS}, we obtain a new description of convex geometries.

\subsection{Representing convex geometries} \label{subsect2d}
Using the usual convex hull operator $\rhullo n$ together with auxiliary points in a tricky way,  Kashiwabara, Nakamura, and  Okamoto~\cite{kashiwabaraatalshelling} gave a representation theorem
for convex geometries in 2005. The  example described in \eqref{eqtxthFgjRcnpmnt} is simpler, 
but it is not appropriate to represent \emph{every} convex geometry because of a very simple reason: if $\pvpair E\Phi$ is a convex geometry of the form described in \eqref{eqtxthFgjRcnpmnt}, then 
$J(\Lat{\pvpair E\Phi})$ is an antichain, which is not so for every convex geometry. Hence, Cz\'edli~\cite{czgcircles} introduced the following construction.  Let $E$ be a finite set of circles in the plane $\preal$, and define a convex geometry $\pvpair E\Phi$ where
$\Phi\colon\pow E\to\pow E$ is defined by 
\begin{equation}
\Phi(X):=\Bigl\{C\in E: C\subseteq \rhullo 2\Bigl(\bigcup_{D\in X} D\Bigr)\Bigr\}.
\label{eqdzhGbRcrtlPs}
\end{equation}
It is easy to see that we obtain a convex geometry in this way. Note, however, that \eqref{eqdzhGbRcrtlPs} does not yield a convex geometry in general if, say, $E$ is a set of triangles rather than a set of circles. 
After translating the problem to lattice theory with the help of Proposition~\ref{propzBjS} and using the toolkit developed for slim semimodular lattices in the papers mentioned in Subsection~\ref{subsect2a}, Cz\'edli~\cite{czgcircles} proved that
\begin{equation}
\left\{\pbxskip
\parbox{9.5cm}{every convex geometry of convex dimension at most $2$ can represented by circles in the sense of \eqref{eqdzhGbRcrtlPs}.}\right.
\label{eqpbxTmPvMqcBp}
\end{equation}
In fact, \cite{czgcircles} proves a bit more. 
While \cite{czgcircles} is mainly a lattice theoretical paper, 
it was soon followed by two results with proofs that are geometrical.
First, Richter and Rogers~\cite{richterrogers} represented every convex geometry analogously to \eqref{eqdzhGbRcrtlPs} but using polygons instead of circles. Second, Cz\'edli and Kincses~\cite{czgkj} replaced polygons with objects taken from an appropriate family of so-called ``almost circles''. However, it was not known at that time whether circles would do instead of ``almost circles".

\subsection{Some results of  geometrical nature}
The problem whether every convex geometry can be represented by circles in the sense of \eqref{eqdzhGbRcrtlPs} was solved in negative by Adaricheva and Bolat~\cite{kabolat}. The main step in their argument is the proof of \cite[Theorem 3.1]{kabolat}; see Corollary~\ref{corolmain} here. In fact, they proved that \eqref{eqpbxThmlnygrSmzWrB}, with self-explanatory syntactical refinements, holds even for arbitrary three \emph{circles} $A_0, A_1, A_2$ and two additional circles, $\,\hal U_0,\hal U_1$. 
This is such an obstacle that does not allow to represent every  convex geometry by circles.  
Even more is true;  later, Kincses~\cite{kincses} found an Erd\H os--Szekeres type obstruction for representing convex geometries by \emph{ellipses}. Similarly to ellipses, he could exclude many other shapes. On the positive side,
Kincses~\cite{kincses} proved that every convex geometry can be represented by \emph{ellipsoids} in $\real^n$ for some $n\in\nplu:=\set{1,2,3,\dots}$ in the sense of \eqref{eqdzhGbRcrtlPs} with $\rhullo n$ instead of $\rhullo 2$. 
However, it is not known whether $n$-dimensional balls could do instead of ellipsoids.

An earlier attempt to generalize Adaricheva and Bolat~\cite[Theorem 3.1]{kabolat}, see Corollary~\ref{corolmain} here, did not use  homotheties and resulted in a new characterization of disks. Namely, for a convex compact set $\,\hal U_0\subseteq \preal$,  Cz\'edli~\cite{czgCharc} proved that
\begin{equation}
\left\{\pbxskip
\parbox{10.0cm}{$\hal U_0$ is a disk if and only if for every isometric copy $\,\hal U_1$ of $\,\hal U_0$ and for any points $A_0,A_1,A_2\in\preal$, \eqref{eqpbxThmlnygrSmzWrB} holds.}\right.
\label{eqpbxChRcrCzg}
\end{equation}
The condition on $\,\hal U_1$ above means that there exists a  distance-preserving geometric transformation $\phi\colon\preal\to\preal$ such that $\,\hal U_1=\phi(\hal U_0)$. 

There are quite many known characterizations of circles and disks;  we mention only one of them below. 
We say that $\,\hal U_0$ and $\,\hal U_1$ are \emph{Fejes-T\'oth crossing} if none of the sets $\,\hal U_0\setminus\hal U_1$ and 
$\hal U_1\setminus\hal U_0$ is path-connected.  It was proved in 
Fejes-T\'oth~\cite{fejestoth} that 
\begin{equation}
\left\{\pbxskip
\parbox{8.5cm}{a convex compact set $\,\hal U_0\subseteq \preal$ is a disk if and only if there is no  isometric copy $\,\hal U_1$ of $\,\hal U_0$ such that  $\,\hal U_0$ and $\,\hal U_1$ are Fejes-T\'oth crossing.}\right.
\label{eqpbxChRcFTW}
\end{equation}
Motivated by the proof of \eqref{eqpbxChRcrCzg}, a more restrictive concept of crossing was introduced in Cz\'edli~\cite{czgCrossing};  it is based on properties of common supporting lines but we will not  define it here. Replacing Fejes-T\'oth crossing with ``\cite{czgCrossing}-crossing'', \eqref{eqpbxChRcFTW} turns into a stronger statement.

Finally, to conclude our mini-survey from George Gr\"atzer's congruence lattices to geometry via a sequence of closely connected 
consecutive results, we note that 
Paul  Erd\H{o}s and E.\ G.\ Straus~\cite{erdosstrauss} extended \eqref{eqpbxChRcFTW} to an analogous characterization of balls in higher dimensions, but the ``\cite{czgCrossing}-crossing'' seems to work only in the plane~$\preal$.

\section{Proofs for the edge-free case}\label{sect3}
As usual in lattice theory, $\,\hal U\subset \hal V$ means the conjunction of $\,\hal U\subseteq \hal V$ and $\,\hal U\neq \hal V$. If $\hal V$ is a \emph{compact} subset of $\preal$, then we often write $\hal V\subset \preal$ since $\hal V\neq\preal$ holds automatically. For compact sets $\,\hal U,\hal V\subseteq\preal$, 
\[\rhullo 2\bigl( \hal U\cup \hal V)= \rhullo 2\bigl(\rhullo 2(\hal U)\cup \hal V\bigr).
\] 
Hence, the inclusion in the third line of \eqref{eqpbxThmlnygrSmzWrB} is equivalent to
the inclusion 
\[
\rhullo 2 (\hal U_{1-k}) \subseteq \rhullo 2\bigl(\rhullo 2(\hal U_k)\cup (\set{A_0,A_1, 
A_2}\setminus\set{A_j})\bigr).
\]
Also, if $\,\hal U$ is compact, then so is $\rhullo2(\hal U)$; see, for example, the first sentence of the introduction in  H\"usseinov~\cite{husseinov}. Thus,  it suffices to prove our theorem only for \emph{convex} compact subsets of $\preal$. Therefore, in the rest of the paper, 
\begin{equation}
\left\{\pbxskip
\parbox{6.7cm}{we will always assume that $\,\hal U$, $\,\hal U_0$, and $\,\hal U_1$ are compact and \emph{convex},}\right.
\label{eqtxtmndgknvx}
\end{equation}
even if this is not repeated all the time.

The advantage of assumption \eqref{eqtxtmndgknvx} lies in the fact that  the properties of planar convex compact sets are well understood. For example, if $ \hal U\subset \preal$ is such a set, then  the boundary $\partial \hal U$ of $\,\hal U$ is known to be a simple closed continuous rectifiable curve; see Latecki, Rosenfeld, and Silverman \cite[Thm. 32]{latecki} and Topogonov \cite[page 15]{topogonov}.
Since the reader need not be a geometer, we note that all what we need to know about planar convex sets are surveyed in a short section of the open access paper Cz\'edli and Stach\'o~\cite{czgstacho}. Some facts about these sets, however, are summarized in the next subsection for the reader's convenience.

\subsection{Supporting lines and a comparison with the case of circles} 
Let 
\begin{equation}
\text{$\Cunit$ denote the \emph{unit circle} 
$\set{\pair x y: x^2+y^2=1}$;}
\label{eqtxtCunit}
\end{equation}
its elements are called \emph{directions}. 
In the rest of the paper,  we often assume that the 
lines $\ell$ in our considerations are \emph{directed lines}; their directions are denoted by $\dir\ell\in\Cunit$ and by  arrows in our figures. A directed line $\ell$ determines two \emph{closed} halfplanes; their intersection is $\ell$. A subset of $\preal$ is \emph{on the left} of $\ell$ if each of its points belongs to the left closed halfplane. Points in the left halfplane of $\ell$ but not on $\ell$ are \emph{strictly on the left} of $\ell$; points \emph{strictly on the right} of $\ell$ are defined analogously.
We always assume that 
\begin{equation}
\left\{\pbxskip
\parbox{7.9cm}{
a supporting line of a set $\,\hal U$ is directed, and it is directed so that $\,\hal U$ is on its left.}\right.
\label{eqpbxSpRtlFt}
\end{equation}
Since every compact convex set in the plane is well known to be  the intersection of the left halfplanes of its supporting lines, we have that 
\begin{equation}
\left\{\pbxskip
\parbox{9.3cm}{if a point $P\in\preal$ does \emph{not} belong to a compact convex set $\,\hal U$, then $\,\hal U$ has a directed supporting line $\ell$ such that $P$ is strictly on the right of $\ell$.}\right.
\label{eqpbxSpRtlNjbbr}
\end{equation} %
If $\ell$ is a supporting line of a compact convex set $\,\hal U$, then the points of $\,\hal U\cap \ell$
are called \emph{support points}. If $\ell$ is the only directed supporting line through a support point $P\in\hal U\cap\ell$, then $\ell$ is a \emph{tangent line} and  $P$  is a \emph{tangent point}. Otherwise, we say that $P$ is a \emph{vertex} of $\,\hal U$. The properties of directed supporting lines are summarized in the open access papers Cz\'edli~\cite{czgCharc} and Cz\'edli and Stach\'o~\cite{czgstacho}, or in the more advanced treatise Bonnesen and Fenchel~\cite{bonfench}.  
In particular, by a \emph{pointed supporting line} of $\,\hal U$ we mean a pair $\pair P\ell$ such that $\ell$ is a directed supporting line of $\,\hal U$ with support point $P$.
In general, $\,\hal U$ may have pointed supporting lines  $\pair {P_1}\ell$ and $\pair {P_2}\ell$ with the same line component but distinct support points $P_1\neq P_2$.

In Cz\'edli~\cite{czganeasy}, which is devoted only to circles, there is a relatively short proof of Adaricheva and Bolat~\cite[Theorem 3.1]{kabolat}, cited as Corollary~\ref{corolmain} here.
Most ideas of \cite{czganeasy} are used in the present paper,
but these ideas need substantial changes in order to overcome the following three difficulties: as opposed to circles, a compact convex set need not have a center with nice geometric properties, its boundary  need not have a tangent line at each of its points, and the boundary  can  include straight line segments of positive lengths. In this section, we disregard the latter difficulty by  calling a compact convex set $\,\hal U$ 
\emph{edge-free} if  no line segment of positive length is a subset of $\partial \hal U$. Equivalently, a  compact convex set $\,\hal U\subset \preal$ is said to be \emph{edge-free} if $\ell\cap \hal U$ is a singleton (still equivalently, if $\ell\cap \partial \hal U$ is a singleton) for every supporting line $\ell$ of $\,\hal U$. 
Note that every singleton subset of $\preal$ is an edge-free compact convex set. Let us emphasize that an edge-free set is \emph{nonempty} by definition. In order to shed even more light on the concept just introduced, we formulate and prove an easy lemma.

\begin{lemma}\label{lemmaNoThree} A nonempty compact convex set $\,\hal U$ is edge-free if and only if $\ell\cap \partial \hal U$ consists of at most two points for every line $\ell$. 
\end{lemma}

\begin{proof} We can assume that $\,\hal U$ is not a singleton since otherwise the statement is trivial. 

First, assume that $\,\hal U$ is \emph{not} edge-free, and pick a supporting line $\ell$ of $\,\hal U$ with two distinct points,  $P_1,P_2\in \ell\cap\partial U$. Let $P_3=(P_1+P_2)/2$; it belongs to $\,\hal U$ by convexity. Since $P_3$
lies on a supporting line, it is not in the interior of $\,\hal U$. Hence,  $P_1,P_2,P_3\in \ell\cap \partial \hal U$, which shows that $\ell\cap \partial \hal U$ consists of more than two points; this implies the ``if'' part of the lemma. 

Second, assume that $\,\hal U$ is edge-free and $\ell$ is a directed line in the plane; we need to show that $\ell\cap \partial \hal U$ consists of at most two points. Suppose the contrary, and let $P_1$, $P_2$, and $P_3$ be three distinct points of $\ell$, in this order, such that they all belong also to $\partial \hal U$.
Pick a supporting line $\ell_2$ of $\,\hal U$ through $P_2$.
Since $\,\hal U$ is edge-free, $\ell_2\cap\partial\hal U$ is a singleton, whereby none of $P_1$ and $P_3$ lies on $\ell_2$.  
Therefore, since  $P_2$ is between $P_1$ and $P_3$, we have that $P_1\in \hal U$ and $P_3\in \hal U$ are strictly on different sides of $\ell_2$; contradicting \eqref{eqpbxSpRtlFt}.  
\end{proof}

Our target in the present section is to prove the following lemma.

\begin{lemma}[Main Lemma]\label{lemmaEFmain}
If the points  $A_0,A_1,A_2\in\preal$ and the \emph{convex} compact sets $\,\hal U_0, \hal U_1\subset \preal$ from Theorem~\textup{\ref{thmmain}} satisfy at least one of the conditions \eqref{thmmaina}, \eqref{thmmainb}, and \eqref{thmmainc} given in the theorem and, in addition,  
\begin{enumerate}[\upshape{\indent}(a)]
\setcounter{enumi}{3}
\item\label{thmmaind} 
$\hal U_0$ and $\,\hal U_1$ are \emph{edge-free},
\end{enumerate}
then implication~\eqref{eqpbxThmlnygrSmzWrB} holds.
\end{lemma}

The proof of this lemma needs some preparation and auxiliary lemmas. In the rest of this section, we always assume that $\,\hal U_0$ and $\,\hal U_1$ are \emph{edge-free}.

\subsection{Comets} In this paper, 
the Euclidean distance $((P_x-Q_x)^2+(P_y-Q_y)^2)^{1/2}$ of $P,Q\in\preal$ is denoted by $\dist P Q$. For nonempty \emph{compact} sets $\,\hal U,\hal V\subset \preal$, 
$\dist {\hal U}{\hal V}=\inf\set{\dist P Q: P\in \hal U,\,\, Q\in \hal V}
=\min\set{\dist P Q: P\in \hal U,\,\, Q\in \hal V}$. For an edge-free compact convex set $\,\hal U$ with more than one elements and a point $F\in \preal\setminus \hal U$, we define the \emph{comet} $\comet F {\hal U}$ with \emph{focus} $F$ and \emph{nucleus} $\,\hal U$ so that
\begin{figure}[htb] 
\centerline
{\includegraphics[scale=1.1]{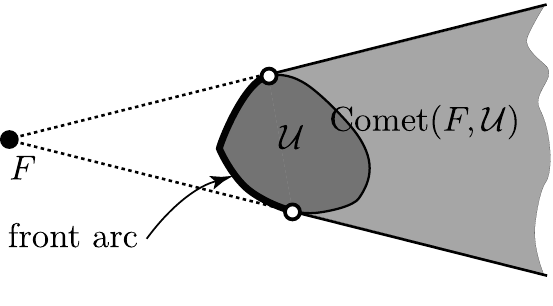}}
\caption{A comet
\label{figegy}}
\end{figure}%
\begin{equation}
\comet F {\hal U}\text{ is the  grey-filled area in Figure~\ref{figegy}.}
\label{eqcrlDsKc}
\end{equation}
More precisely, if we consider $F$ as a source of light, then  $\comet F {\hal U}$ is the topological closure of the set of points that are shadowed by the nucleus $\,\hal U$. Note that 
$\hal U$, which is dark-grey in the figure, is a subset of $\comet F {\hal U}$ and we have that $\dist {\set F}{\comet F {\hal U}}>0$.  
As opposed to $\,\hal U$,  $\comet F {\hal U}$ is never compact.

Since $\,\hal U$ is compact, convex, and not a singleton, there are exactly two supporting lines of $\,\hal U$ through $F$, and they are supporting lines of $\comet F {\hal U}$ as well. Since $\,\hal U$ is edge-free, these two lines are \emph{tangent lines} of $\,\hal U$ and also of $\comet F {\hal U}$.
Each of these tangent lines has a  unique tangent point on $\partial \hal U$. The arc of $\partial \hal U$ between these points that is closer to $F$ is the \emph{front arc} of the comet; see the thick curve in Figure~\ref{figegy}.
Note that the boundary of $\comet F {\hal U}$ is the union of the front arc and two half-lines, so comets are never edge-free.

\subsection{Externally perspective compact convex sets}
For topologically closed convex sets $\hal V_1,\hal V_2\subseteq \preal$, we will say that 
\begin{equation}
\text{$\hal V_1$ is \emph{loosely included} in $\hal V_2$, in notation, $\hal V_1\loosin \hal V_2$,}
\label{eqlooSinCl}
\end{equation}
if every point of $\hal V_1$ is an internal point of $\hal V_2$. 
The \emph{interior} of a compact convex set $\,\hal U$ will be denoted by $\inter {\hal U}$; note that $\inter {\hal U}=\hal U\setminus\partial \hal U$. 
Clearly, if $\hal V_1\subset \preal$ is compact, $\hal V_2\subseteq \preal$ is closed, and $P\in\inter {\hal V_1}$, then  
\begin{equation}
\left\{\pbxskip
\parbox{8.4cm}{$\hal V_1\loosin \hal V_2$ implies that there is a $\delta>0$ such that $\hthet P{1+\epsilon}(\hal V_1)\loosin \hal V_2$ for all positive $\epsilon\leq \delta$,}\right.
\label{eqpbxTlSnHfgjQwpLnt}
\end{equation}
because  $\preal\setminus \inter {\hal V_2}$  is closed and its distance from $\hal V_1$ is positive.

Next, for  compact convex sets $\,\hal U_1,\hal U_2\subset \preal$, each of them with more than one element, we say that  $\,\hal U_1$ and $\,\hal U_2$ are \emph{externally perspective} if 
$\hal U_2=\hthet P\lambda(\hal U_1)$ for some (in fact, unique) $0<\lambda\in \real\setminus\set 1$ and $P\in \preal\setminus\rhullo2(\hal U_1\cup \hal U_2)$; see \eqref{eqtxthmThdF}. Equivalently, $\,\hal U_1$ and $\,\hal U_2$ are \emph{externally perspective}
if $\,\hal U_2=\hthet P\lambda(\hal U_1)$ 
with $P\notin \hal U_1$ and $0<\lambda\neq 1$. Hence, by interchanging the subscripts if necessary, we will often assume that $0<\lambda <1$ if $\,\hal U_2=\hthet P\lambda(\hal U_1)$ is externally perspective to $\,\hal U_1$.

\begin{figure}[ht] 
\centerline
{\includegraphics[scale=1.1]{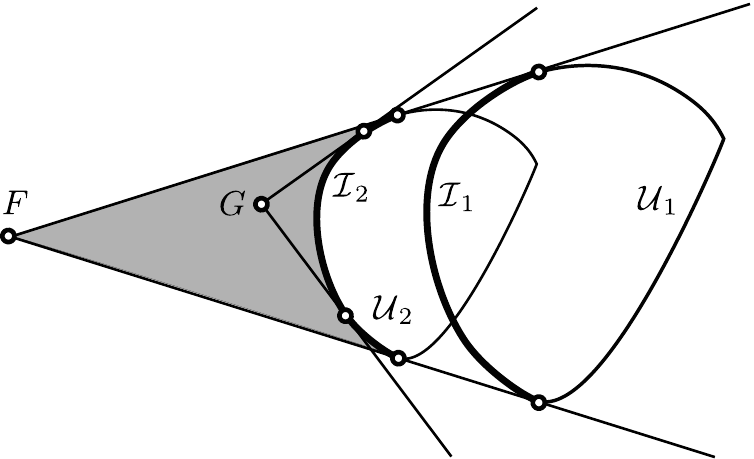}}
\caption{Illustration for Lemma~\ref{lemmaperspCs}
\label{figmbAT}}
\end{figure}%

The following lemma is obvious by Figure~\ref{figmbAT}.

\begin{lemma}\label{lemmaperspCs}
Let $\,\hal U_1$ and $\,\hal U_2=\hthet{F}\lambda(\hal U_1)$ be externally perspective compact convex subsets of the plane
such that  $0<\lambda<1$.  
If $G$ is an internal point of the grey-filled area surrounded by the common tangent lines of $\,\hal U_1$ and $\,\hal U_2$ through $F$ and the front arc $\hal I_2$ of $\comet F{\hal U_2}$, then   $\comet F{\hal U_1}$ is loosely included in $\comet G{\hal U_2}$.
\end{lemma}

In the rest of the paper, to ease the notation, 
\begin{equation}
\text{$\hsz$ will stand for $\rhullo 2(\set{A_0,A_1,A_2})$.}
\label{eqpbxZhFtT}
\end{equation}
Next, as a ``loose counterpart'' of the 2-Carousel Rule defined in  Adar\-icheva~\cite{adaricarousel}, we formulate the following lemma.

\begin{lemma}\label{lemmafpnTlsl}
Let $A_0$, $A_1$, and $A_2$ be non-collinear points in the plane. If  $B_0$ and $B_1$ are \emph{distinct internal} points of $\hsz$, then there exist $j\in\set{0,1,2}$ and $k\in\set{0,1}$ such~that 
\[\set{B_{1-k}}\loosin \rhullo 2\bigl(\set{B_k}\cup (\set{A_0,A_1,A_2}\setminus\set{A_j})\bigr).
\]
\end{lemma}

\begin{figure}[ht] 
\centerline
{\includegraphics[scale=1.1]{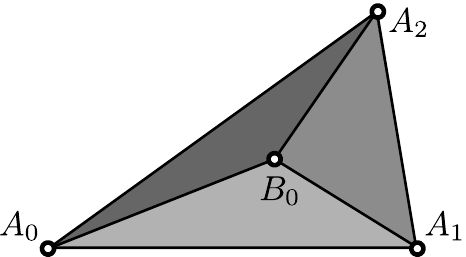}}
\caption{Illustration for the proof of Lemma~\ref{lemmafpnTlsl}
\label{figtcarou}}
\end{figure}%

\begin{proof} If $B_1$ is in the interior of one of the three little triangles that are colored with different shades of grey
in Figure~\ref{figtcarou}, then we can let $k:=0$. Otherwise, $B_1$ is an internal point of  one of the line segments $[A_0,B_0]$, $[A_1,B_0]$, and $[A_2,B_0]$, and we can let $k:=1$. In both cases, it is clear that we can choose an appropriate $j\in\set{0,1,2}$.
\end{proof}

\begin{lemma}\label{lemmasinglton}  
Condition \textup{\ref{thmmain}}\eqref{thmmainc}, even without assuming \textup{\ref{lemmaEFmain}}\eqref{thmmaind},  implies \eqref{eqpbxThmlnygrSmzWrB}, that is,  the conclusion of Lemma~\textup{\ref{lemmaEFmain}}.
\end{lemma}

\begin{proof}  
Since $\,\hal U_0$ and $\,\hal U_1$ play symmetric roles, we can assume that $\,\hal U_0=\set{B_0}$ is a singleton. 
We can assume also that $B_0\in\inter{\hsz}$, because otherwise the statement is trivial. 
If there exists a point $B_1\in \hal U_1$ and a subscript $j\in\set{0,1,2}$ such that 
\begin{equation}
B_0\in \rhullo 2\bigl(\set{B_1}\cup (\set{A_0,A_1,A_2}\setminus\set{A_j})\bigr),
\label{eqZhbRtPfjWjMtDnMh}
\end{equation} 
then \eqref{eqpbxThmlnygrSmzWrB} holds with $k=1$ and this $j$. 
So, we assume that \eqref{eqZhbRtPfjWjMtDnMh} fails for all $B_1\in\hal U_1$ and all $j\in\set{0,1,2}$. Then, by Lemma~\ref{lemmafpnTlsl} if $B_1$ below is in $\inter{\hsz}$ or trivially if $B_1\in\partial\hsz$,  
\begin{equation}
\left\{\pbxskip
\parbox{10.7cm}{
for each $B_1\in \hal U_1$, there is a smallest $j=j(B_1)\in\set{0,1,2}$
such that 
$B_1\in \rhullo 2\bigl(\set{B_0}\cup (\set{A_0,A_1,A_2}\setminus\set{A_{j(B_1)}})\bigr)$.}\right.
\label{eqpbxGzTnMnL}
\end{equation} 
If $j=j(B_1)$ does not depend on $B_1\in\hal U_1$, then \eqref{eqpbxGzTnMnL} gives the satisfaction of \eqref{eqpbxThmlnygrSmzWrB} with $k=0$ and this $j$. For the sake of contradiction, suppose that $j(B_1)$ depends on $B_1\in \hal U_1$. By  \eqref{eqpbxGzTnMnL}, this means that there are points $B_1'$ and $B''_1$ in $\,\hal U_1$ that belong to distinct little triangles (colored by different shades of grey) in Figure~\ref{figtcarou}. By convexity, $[B_1',B_1'']\subseteq \hal U_1$. Hence, $\,\hal U_1$ has a point $B_1$ that belongs to one of the line segments $[A_0,B_0]$, $[A_1,B_0]$, and $[A_2,B_0]$. This $B_1$ shows the validity of \eqref{eqZhbRtPfjWjMtDnMh} for some $j$, which contradicts our assumption that  \eqref{eqZhbRtPfjWjMtDnMh} fails for all $j$. Thus, $j(B_1)$ does not depend on $B_1\in \hal U_1$, completing the proof. 
\end{proof}

\subsection{Internally tangent edge-free compact convex sets}
We say that $\,\hal U_0$ and $\,\hal U_1$ subject to \ref{thmmain}\eqref{thmmaina} or \ref{thmmain}\eqref{thmmainb} are \emph{internally tangent} if they have a common pointed supporting line. For example, as it is shown in Figure~\ref{figinttang}, if 
\begin{figure}[ht] 
\centerline
{\includegraphics[scale=1.1]{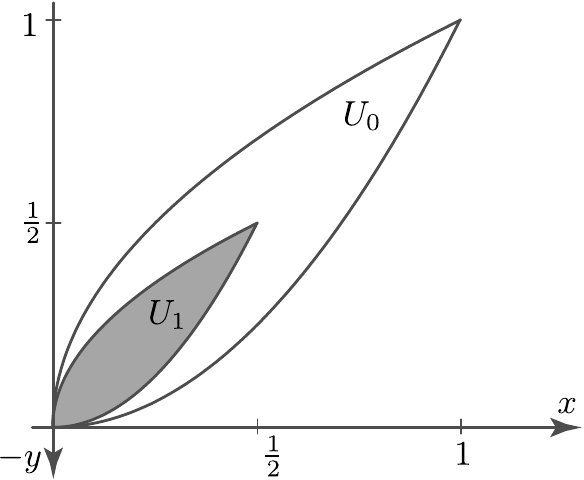}}
\caption{Two internally tangent edge-free compact convex sets
\label{figinttang}}
\end{figure}%
\begin{align}
\hal U_0 &:=\set{\pair x y: 0\leq x\leq 1,\,\,\, x^2\leq y\leq 1-(x-1)^2}\text{ and }\cr
\hal U_1 &:=\hthet{\pair 0 0}{1/2}(\hal U_0),
\label{eqpRbLhmtntbb}
\end{align}
then $\,\hal U_0$ and $\,\hal U_1$ are internally tangent edge-free compact convex sets. Let $O=\pair 0 0$.  Denoting the  abscissa axis with the usual orientation $\pair 1 0\in \Cunit$ and the ordinate axis with the unusual reverse orientation $\pair0{-1}\in\Cunit$ by $x$ and $-y$, respectively, both $\pair {O}x$ and $\pair O{-y}$ are common pointed supporting lines of $\,\hal U_0$ and $\,\hal U_1$. This shows that 
condition \ref{thmmain}\eqref{thmmaina} together with 
\ref{lemmaEFmain}\eqref{thmmaind}
 do not imply the uniqueness of the common supporting lines through a point of $\partial \hal U_0\cap\partial \hal U_1$ if $\,\hal U_0$ and $\,\hal U_1$ are internally tangent. In case of \eqref{eqpRbLhmtntbb} and similar cases, these pointed supporting lines have the same support point and $\,\hal U_0$ and $\,\hal U_1$ are tangent to each other in some sense.
The aim of this subsection is to prove the following lemma.

\begin{lemma}\label{lemmainttang} If $\,\hal U_0$ and $\,\hal U_1$ are non-singleton, internally tangent, edge-free compact convex subsets of $\,\preal$, then the following two assertions hold.
\begin{enumerate}[\upshape\quad(i)]
\item\label{lemmainttanga} If $\,\hal U_1=\hthet P\lambda(\hal U_0)$ for some $0<\lambda \in\real$ and $P\in \preal$, as in \textup{\ref{thmmain}\eqref{thmmaina}}, then either $\,\hal U_1=\hal U_0$ and $\lambda=1$, or  $\lambda\neq 1$ and  $\partial \hal U_0\cap \partial \hal U_1=\set{P}$.  Furthermore, if $\lambda>1$ then $\,\hal U_1\supseteq \hal U_0$ while $0<\lambda<1$ implies that $\,\hal U_1\subseteq \hal U_0$.
\item\label{lemmainttangb} If $\,\hal U_1$ is obtained from $\,\hal U_0$ by a translation as in \textup{\ref{thmmain}\eqref{thmmainb}}, then $\,\hal U_1=\hal U_0$ and the translation in question is the identity map.
\end{enumerate}
\end{lemma}

Note that this lemma fails without assuming that $\,\hal U_0$ and $\,\hal U_1$ are edge-free. To exemplify this, let $\,\hal U_0$ be the rectangle
$\set{\pair x y: -2\leq x \leq 2\text{ and } 0\leq y\leq 2}$. 
Then $\,\hal U_1:=\hthet{\pair 4 2}{1/2}(\hal U_0)$ and $\,\hal U_1:=\set{\pair{x+1}y: \pair x y\in \hal U_0}$ would witness the failure of 
 \ref{lemmainttang}\eqref{lemmainttanga} and that of  \ref{lemmainttang}\eqref{lemmainttangb}, respectively.

\begin{proof}  Let $\pair {P^\ast}{\ell^\ast}$ be a common pointed supporting line of $\,\hal U_0$ and $\,\hal U_1$.
  
First, assume that $\,\hal U_1=\hthet P\lambda(\hal U_0)$ as in \eqref{lemmainttanga}. We can assume that $\,\hal U_0\neq \hal U_1$ since otherwise the lemma is trivial.
So we know that $0<\lambda\neq 1$. Since $\pair {P^\ast}{\ell^\ast}$ is a pointed supporting line of $\,\hal U_0$, so is $\pair{P'}{\ell'}:=\pair{\hthet P\lambda(P^\ast)}{\hthet P\lambda(\ell^\ast)}$ of $\,\hal U_1=\hthet P\lambda(\hal U_0)$. We have that $\dir{\ell'}=\dir{\ell^\ast}$; note that this is one of the reasons that $\lambda>0$ is always assumed in this paper. It is well known from the folklore that 
for each $\alpha\in\Cunit$ and every compact convex set $\,\hal U$, 
\begin{equation}
\text{$\hal U$ has exactly one directed supporting line of  direction $\alpha$;}
\label{eqpbxZhgTnQspL}
\end{equation}
see Bonnesen and Fenchel~\cite{bonfench},  Yaglom and Boltyanski\v\i~\cite[page 8]{yagbolt}, or Cz\'edli and Stach\'o~\cite{czgstacho}. Hence, $\ell'$ and $\ell^\ast$ are the same supporting lines of $\,\hal U_1$. Since $\,\hal U_1$ is edge-free,
$\ell^\ast=\ell'$ has only one support point, whence $P^\ast=P'$. So $P^\ast=P'=\hthet P\lambda(P^\ast)$. Since $\lambda\neq 1$, the homothety $\hthet P\lambda$ has only one fixed point, whereby $P^\ast=P$, as required.
Next, let $Q$ be an arbitrary element of $\partial \hal U_0\cap\partial \hal U_1$. For the sake of contradiction, suppose that $Q\neq P$. Since $\lambda\neq 1$, the \emph{collinear} points $P$, $Q$, and $Q':=\hthet P\lambda(Q)$ are pairwise distinct. Since the $\hthet P\lambda$-image of a boundary point is a boundary point, these three collinear points belong to $\partial \hal U_1$. This contradicts Lemma~\ref{lemmaNoThree} and  settles the first sentence of \eqref{lemmainttanga}.

Since $\hthet P{1/\lambda}$ is the inverse of $\hthet P\lambda$, it suffices to prove the second sentence of 
 \eqref{lemmainttanga} only for $\lambda>1$, because then the case $0<\lambda<1$ will follow by replacing $\tuple{\hal U_0, \hal U_1,\lambda}$ by $\tuple{\hal U_1, \hal U_0,1/\lambda}$. So, let $X$ be an arbitrary point of $\,\hal U_0$. Using that $X\in\rhullo 2(\set{P, \hthet {P}\lambda(X)})$, $P=P^\ast=P'\in \hal U_1$, and $\hthet {P}\lambda(X)\in \hthet {P}\lambda(\hal U_0)=\hal U_1$, the convexity of $\,\hal U_1$ implies that $X\in \hal U_1$, as required. This completes the proof of part \eqref{lemmainttanga}.

The argument for \eqref{lemmainttangb} is similar.
Let $\phi$ denote the translation such that $\,\hal U_1=\phi(\hal U_0)$.  Let $\pair{P'}{\ell'}:=\pair{\phi(P^\ast)}{\phi(\ell^\ast)}$. As in the previous paragraph, we obtain that  $\pair{P'}{\ell'}$ and $\pair{P^\ast}{\ell^\ast}$ are both supporting lines of $\,\hal U_1$. Since $\dir{\ell'}=\dir{\ell^\ast}$, we have that $\ell'=\ell^\ast$. Thus,  using that $\,\hal U_1$ is edge-free, we obtain that $P'=P^\ast$. So $P'=\phi(P^\ast)$ is a fixed point of the translation $\phi$. Hence, $\phi$ is the identity map, and we conclude that
$\hal U_1=\phi(\hal U_0)=\hal U_0$, as required.
\end{proof}

\subsection{Technical lemmas}

We compose maps from right to left, so note the rule  $(\phi\circ \psi)(x)=\phi\bigl(\psi(x)\bigr)$. The first technical lemma we need is the following.

\begin{lemma}\label{lemmaBlowUp} 
Let $\phi\colon \preal\to\preal$ be  a homothety $\hthet P\lambda$  or a translation, let $P_0\in \preal$ be a point, and let $P_1=\phi(P_0)$. Then for every 
$\xi\in\real\setminus\set0$, 
\[ \phi \circ \hthet {P_0}\xi = \hthet {P_1}\xi \circ \phi \,\,
\text{ or, equivalently, }\,\,\hthet {P_1}\xi = \phi \circ \hthet {P_0}\xi\circ \phi^{-1}.
\]
\end{lemma}

\begin{proof} Since 
$ \phi \circ \hthet {P_0}\xi\circ \phi^{-1}$ is clearly a homothety of ratio $\xi$ that fixes ${P_1}$,  this homothety is $\hthet {P_1}\xi$, as required.  
\end{proof}

The next technical lemma will also be needed. It  follows by straightforward computation  with the help of computer algebra; an  appropriate worksheet for Maple V Release 5 is available from the homepage of the first author.
After stating the lemma, we give a more geometrical and \emph{short} proof.

\begin{figure}[htb] 
\centerline
{\includegraphics[scale=1.1]{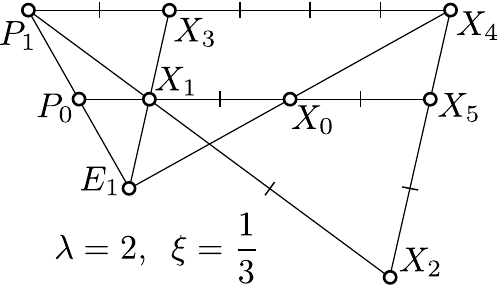}}
\caption{An illustration of Lemma~\ref{lemmaTechnmwsgT}
\label{figtLmsthrj}}
\end{figure}% 

\begin{lemma}\label{lemmaTechnmwsgT}
Let $\lambda,\xi\in \real\setminus\set 0$, let ${E_1}, P_0, X_0\in \preal$, and define the  points
\begin{align*}
&(1)\lmskip P_1:=\hthet {E_1}{\lambda}(P_0),&&(2)\lmskip  X_1:=\hthet{P_0}\xi(X_0),
&& (3) \lmskip  X_2:=\hthet{P_1}{1/\xi}(X_1),\cr
&(4)\lmskip  X_3:=\hthet {E_1}{\lambda}(X_1), &&(5)\lmskip    X_4:=\hthet{P_1}{1/\xi}(X_3) &&(6)\lmskip   X_5:=\hthet {X_4}\xi (X_2).
\end{align*}
Then $\hthet{X_1}{1/\lambda}(X_5) = X_0$.  
\end{lemma}

\begin{proof} 
If $\xi=1$ or $\lambda=1$, then the statement is obvious.
If $\xi\ne1\ne\lambda$, then  Figure~\ref{figtLmsthrj} visualizes what we have.
By (3) and (5) we deduce $\xi\overrightarrow{X_4X_2}=\overrightarrow{X_3X_1}$,
hence (6) gives $\overrightarrow{X_4X_5}=\overrightarrow{X_3X_1}$,
so quadrangle $\tuple{X_1,X_3,X_4,X_5}$ is a parallelogram.
Thus,
\[
\overrightarrow{X_1X_5}=\overrightarrow{X_3X_4}
 \overset{(5)}=\frac{\xi-1}{\xi}\overrightarrow{X_3P_1}
 \overset{(1),(4)}=\frac{\lambda(\xi-1)}{\xi}\overrightarrow{X_1P_0}
 \overset{(2)}=\lambda\overrightarrow{X_1X_0}.\qedhere
\]
\end{proof}

The following lemma is well known, especially without the adjective ``positive''. However, there are other variants and the corresponding terminology is not unique in the literature;
for example,  Schneider~\cite[page xii]{Schneider1993}
includes translations in the concept of positive homotheties. 
The terminological ambiguity in the literature justifies that we formulate this lemma and give its trivial proof.

\begin{lemma}\label{lemmafrcSPpRtT} Let $G$ be the collection of all positive homotheties and all translations of the plane. Then $G$ is a group with respect to composition. 
\end{lemma}

\begin{proof}  It suffices to show that if $\phi_1$ and $\phi_2$ belongs to $G$, then so does $\phi:=\phi_1\circ\phi_2$. Since $\phi_1$ and $\phi_2$ are similarity transformations that preserve the directions of directed lines, the same holds for $\phi$. This implies that $\phi\in G$.
\end{proof}

\subsection{The lion's share of the proof}
First, we prove the following lemma.

\begin{lemma}\label{lemmangszrnwnKztlS} 
Assume that $\,\hal U_0$ is an edge-free compact convex subset of $\preal$, $\phi\colon \preal\to\preal$ is a positive homothety or a translation, $\,\hal U_1=\phi(\hal U_0)$,
 $P_0\in\inter{\hal U_0}$, $P_1=\phi(P_0)$, $\xi\in (0,1)=[0,1]\setminus\set{0,1}$, $\,\hal U_0(\xi)=\hthet{P_0}\xi(\hal U_0)$, $\,\hal U_1(\xi)=\hthet{P_1}\xi(\hal U_1)$, and $\,\hal U_0(\xi)$ and $\,\hal U_1(\xi)$ are internally tangent. Then at least one of the following two assertions hold:
\begin{enumeratei}
\item \label{lemmangszrnwnKztlSa} $\,\hal U_0\subseteq \hal U_1$ and  $\,\hal U_0(\xi)\subseteq \hal U_1(\xi)$, or
\item\label{lemmangszrnwnKztlSb} $\,\hal U_1\subseteq \hal U_0$ and $\,\hal U_1(\xi)\subseteq \hal U_0(\xi)$.
\end{enumeratei}
\end{lemma}

\begin{proof} We can assume that  $\phi$ is not the identity map, since otherwise the statement trivially holds. Computing by Lemma~\ref{lemmaBlowUp}, we obtain that
\begin{align}
\hal U_1(\xi) &= \hthet{P_1}{\xi}\bigl(\phi(\hal U_0)\bigr)
= (\hthet{P_1}{\xi}\circ \phi\circ \hthet{P_0}{1/\xi})\bigl(\hal U_0(\xi)\bigr) \cr
& \overset {\textup{Lem.\,\ref{lemmaBlowUp}}} = 
(\phi\circ\hthet{P_0}{\xi}\circ \hthet{P_0}{1/\xi})\bigl(\hal U_0(\xi)\bigr) = \phi\bigl(\hal U_0(\xi)\bigr),
\label{alignZTrnVtGhJ}
\end{align}
which shows that Lemma~ \ref{lemmainttang} is applicable to the triplet $\tuple{\hal U_0(\xi),\hal U_1(\xi),\phi}$.
Let $\pair {E_1}\ell$ be a common pointed supporting line of $\,\hal U_0(\xi)$ and $\,\hal U_1(\xi)$.
Since $\phi$ is not the identity map, 
Lemma~ \ref{lemmainttang} gives that
$\phi=\hthet{E_1}\lambda$ for some $\lambda>0$. 
The systems $\tuple{\hal U_0,\hal U_1,\lambda, \phi=\hthet{E_1}\lambda}$ and $\tuple{\hal U_1,\hal U_0,1/\lambda, \phi^{-1}=\hthet{E_1}{1/\lambda}}$ play symmetric roles, whence we can assume that $\lambda\geq 1$. We obtain by Lemma~\ref{lemmainttang} that 
\begin{equation}
\hal U_0(\xi)\subseteq \hal U_1(\xi).
\label{eqdhbzWrnBghRQ}
\end{equation}
In order to prove the inclusion $\,\hal U_0\subseteq \hal U_1$, let $X_0\in \hal U_0$. Since $\phi=\hthet{E_1}\lambda$, 
we have that $P_1 = \hthet{E_1}\lambda(P_0)$. Consider the points $X_1,\dots, X_5$ defined in Lemma~\ref{lemmaTechnmwsgT}. By the definition of $\,\hal U_0(\xi)$, we have that $X_1\in \hal U_0(\xi)$, whereby \eqref{eqdhbzWrnBghRQ} yields that  $X_1\in \hal U_1(\xi)$. Thus, since $\hthet{P_1}{1/\xi}$ is the inverse of $\hthet{P_1}\xi$, the definition of $\,\hal U_1(\xi)$ leads to $X_2\in \hal U_1$. Since $X_1\in \hal U_0(\xi)$, we obtain by equation (4) of Lemma~\ref{lemmaTechnmwsgT}, $\hthet{E_1}\lambda=\phi$, and \eqref{alignZTrnVtGhJ} that $X_3\in \hal U_1(\xi)$. 
This gives that $X_4\in \hal U_1$. 
Since $X_2,X_4\in \hal U_1$ and $0<\xi<1$, the convexity of $\,\hal U_1$ implies that 
$X_5\in \hal U_1$. Using that $X_1\in \hal U_1(\xi)\subseteq \hal U_1$ and that $0< 1/\lambda \leq 1$, the convexity of $\,\hal U_1$ gives that $\hthet{X_1}{1/\lambda}(X_5)\in \hal U_1$. Thus, $X_0\in \hal U_1$ by Lemma~\ref{lemmaTechnmwsgT}, proving that $\,\hal U_0\subseteq \hal U_1$, as required.
\end{proof}

Now, armed with the auxiliary statements proved so far, we are in the position to prove the (Main) Lemma~\ref{lemmaEFmain}.

\begin{proof}[Proof of Lemma~\ref{lemmaEFmain}] 
Lemma~\ref{lemmasinglton} allows us to assume that none of $\,\hal U_0$ and $\,\hal U_1$ is a singleton. For the sake of contradiction, suppose that the lemma fails. Let $\,\hal U_0,\hal U_1,A_0,A_1$, and $A_2$ witness this failure. We can assume that  $A_0$, $A_1$, and $A_2$ is the counterclockwise list of the vertices of triangle ${\hsz}$; see \eqref{eqpbxZhFtT}. 
If \ref{thmmain}\eqref{thmmaina} holds, then $\phi\colon \preal\to\preal$ will denote the transformation $\hthet P\lambda$ mentioned in \ref{thmmain}\eqref{thmmaina}. Similarly, if \ref{thmmain}\eqref{thmmainb} holds, then
$\phi\colon\preal\to\preal$ stands for a translation according to  \ref{thmmain}\eqref{thmmainb}. In both cases, $\,\hal U_1=\phi(\hal U_0)$. Fix an internal point $P_0$ of $\,\hal U_0$, and let $P_1:=\phi(P_0)$. For $i\in \set{0,1}$ and every real number $\xi\in [0,1]$, let 
\begin{equation}
\hal U_i(\xi):=\hthet{P_i}\xi(\hal U_i).
\label{eqVxhsxzTfHh}
\end{equation}
Note that $\hthet{P_i}\xi$ is a positive homothety only for $\xi>0$ but  \eqref{eqtxthmThdF} is meaningful also for $\lambda=0$. In particular, $\,\hal U_i(0):=\hthet{P_i}0(\hal U_i)$ makes sense and it is understood as $\set{P_i}$. For $\xi=0$, we trivially have that $\,\hal U_1(\xi) =\phi \bigl(\hal U_0(\xi)\bigr)$. So let $\xi>0$.
Lemma~\ref{lemmaBlowUp} and 
 \eqref{eqVxhsxzTfHh} 
yield that 
\begin{align}
\hal U_1(\xi) &=\hthet{P_1}\xi (\hal U_1) \cr
&= (\phi\circ\hthet{P_0}\xi)\bigl(\phi^{-1}(\hal U_1)\bigr)
=\phi \bigl(\hthet{P_0}\xi(\hal U_0)\bigr)=\phi(\hal U_0(\xi)).
\label{eqKrPdJvsT}
\end{align}
Thus, 
\begin{equation}
\hal U_1(\xi) =\phi \bigl(\hal U_0(\xi)\bigr),\quad\text{ for all }\xi\in[0,1]\text{ and }i\in\set{0,1}.
\label{eqThTspZnPrmrq}
\end{equation}
Since $\,\hal U_i(\xi)\subseteq \hal U_i(1)=\hal U_i\subseteq {\hsz}$, we have also that $\,\hal U_i(\xi)\subseteq {\hsz}$.

Let $H$ be the set of all $\eta\in[0,1]$ such that \eqref{eqpbxThmlnygrSmzWrB} holds for $\,\hal U_0(\eta)$ and $\,\hal U_1(\eta)$ with some $j$ and $k$. Since $0\in H$ by Lemma~\ref{lemmasinglton}, or even by Lemma~\ref{lemmafpnTlsl}, $H\neq\emptyset$. Hence, $H$ has a supremum, which we denote by $\xi$. 
It follows from Lemma~\ref{lemmafpnTlsl} and continuity that $\xi > 0$. A standard compactness argument 
 shows that $\xi\in H$, that is, $\xi$ is the maximal element of $H$. Since a more involved, similar, but still standard  argument will be given after \eqref{eqlmBdkadhzksllVt}, we do not give the details of this compactness argument here; note that the omitted details, modulo insignificant changes, are given in the extended version, arXiv:1610.02540, of Cz\'edli~\cite{czganeasy}. 
Taking our indirect assumption and $\xi=\max(H)$ into account, we have that
\begin{equation}
0<\xi :=\max(H)<1.
\label{eqtxtxiisMax}
\end{equation}
Since $\xi\in\goodset$, we can assume that the indices are chosen so that, as Figure~\ref{fightmThGdrrgtT} shows, $\,\hal U_1(\xi)$
is included in the grey-filled ``curved-backed trapezoid'' 
\begin{equation}
\trap(\xi):= \rhullo2(\set{A_0,A_1}\cup \hal U_0(\xi)).
\label{eqPxidfnt}
\end{equation}
In Figure~\ref{fightmThGdrrgtT}, the  ``back'' of this trapezoid is the thick curve connecting $E_0$ and $F_0$.  If $\,\hal U_1(\xi)$ was included in the interior of $\trap(\xi)$, then there would be a (small) positive $\epsilon$ such that $\,\hal U_1(\xi+\epsilon )\subseteq \trap(\xi)\subseteq \trap(\xi+\epsilon)$  and $\xi+\epsilon$ would belong to $\goodset$, contradicting the fact that $\xi$ is the largest element of $H$. 
Hence,  $\,\hal U_1(\xi)\subseteq \trap(\xi)$, but the intersection $\partial \hal U_1(\xi)\cap \partial \trap(\xi)$  has at least one point. So we can pick a point
$E_1\in \partial \hal U_1(\xi)\cap \partial \trap(\xi)$. 
 Since $\,\hal U_1(\xi)\loosin \hal U_1(1)=\hal U_1\subseteq {\hsz}$,  we have that $E_1\notin \partial {\hsz}$. Since the ``left leg'', that is, the straight line segment $[A_0,E_0]$, and the ``right leg'' $[A_1,F_0]$ of $\trap(\xi)$ play symmetric roles, it suffices to consider only the following two cases:
either $E_1$ belongs  to the ``back'' of $\trap(\xi)$, including its endpoints $E_0$ and $F_0$, or $E_1$ belongs to the ``left leg'' $[A_0,E_0]$, excluding $E_0$.

\begin{figure}[ht] 
\centerline
{\includegraphics[scale=1.1]{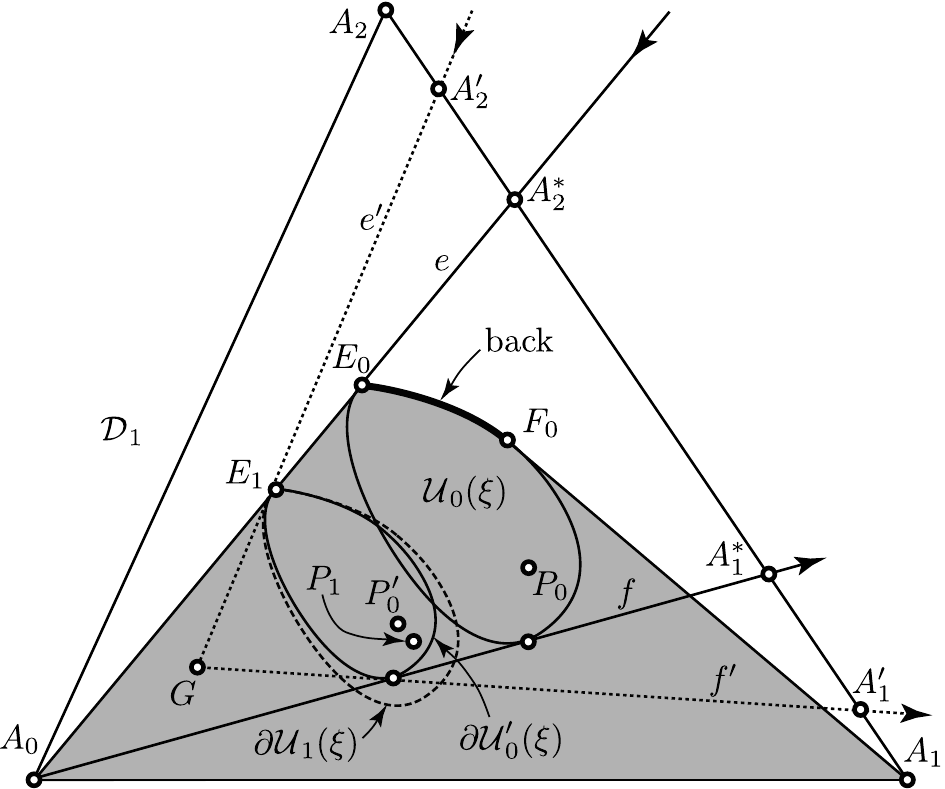}}
\caption{Illustration for the proof of Lemma~\ref{lemmaEFmain} 
\label{fightmThGdrrgtT}}
\end{figure}%

First, assume that $E_1$ belongs to the ``back'' of $\trap(\xi)$. 
Then, clearly, $E_1$ belongs to $\partial\hal U_0(\xi)$. 
Since $\trap(\xi)$ is a compact convex set, it has a directed
supporting line $\ell$ through $E_1$. Since $\,\hal U_0(\xi)\subseteq \trap(\xi)$ and  $\,\hal U_1(\xi)\subseteq \trap(\xi)$, both $\,\hal U_0(\xi)$ and $\,\hal U_1(\xi)$ are on the left of $\ell$. Using that $E_1$ is in $\,\hal U_0(\xi)\cap \hal U_1(\xi)$, it follows that  $\pair {E_1}\ell$ is a common  pointed supporting line of $\,\hal U_0(\xi)$ and $\,\hal U_1(\xi)$. Hence, $\,\hal U_0(\xi)$ and $\,\hal U_1(\xi)$ are internally tangent. Furthermore, it follows from Lemma~\ref{lemmafrcSPpRtT} and \eqref{eqVxhsxzTfHh} that  $\,\hal U_1(\xi)$ is obtained from  $\,\hal U_0(\xi)$ by a translation or a positive homothety. 
Therefore, Lemma~\ref{lemmangszrnwnKztlS} yields that $\,\hal U_0\subseteq \hal U_1$ or $\,\hal U_1\subseteq \hal U_2$. This trivially implies \eqref{eqpbxThmlnygrSmzWrB}, contradicting the initial assumption of the proof. Thus, the first case where $E_1$ belongs to the back of  $\trap(\xi)$ has been excluded.

Second, assume that $E_1$ belongs to the ``left leg'' $[A_0,E_0]$ as illustrated in Figure~\ref{fightmThGdrrgtT}. Since $\,\hal U_1(\xi)\loosin U_1\subseteq \hsz$ implies that  $E_1\notin\partial \hsz$, we have that $E_1\ne A_0$. We can assume that  $E_1\neq E_0$ since the opposite case has already been settled. Hence, letting $\nu:=\dist{A_0}{E_1}/\dist{A_0}{E_0}$, we have that $0<\nu<1$. Let 
\begin{align}
\hal U_0'&:=\hthet{A_0}\nu(\hal U_0),\,\,\text{ }\,\, P_0':=\hthet{A_0}\nu(P_0),\label{eqdhzTrMBQa}
\\
\phi'&:=  \phi \circ \hthet{A_0}{1/\nu},\,\,\text{ and }\,\, \hal U_0'(\xi):=\hthet{P_0'}\xi(\hal U_0').
\label{eqdhzTrMBQb}
\end{align}
The position of $\,\hal U_0'(\xi)$  in Figure~\ref{fightmThGdrrgtT} is justified by 
\begin{align}
&\hthet{A_0}\nu(\hal U_0(\xi))  \eeqref{eqVxhsxzTfHh} (\hthet{A_0}\nu \circ \hthet{P_0}\xi)(\hal U_0)
\cr
&  \overset {\textup{Lem.\,\ref{lemmaBlowUp}}} = 
(\hthet{P'_0}\xi \circ \hthet{A_0}\nu)(\hal U_0) 
 \eeqref{eqdhzTrMBQa}  \hthet{P'_0}\xi (\hal U'_0) 
\eeqref{eqdhzTrMBQb} \hal U'_0(\xi).
\label{alignJsTfvSnpSzrTDb}
\end{align}
Since $\hthet{A_0}{1/\nu}$ is the inverse of $\hthet{A_0}{\nu}$,
\eqref{eqdhzTrMBQa} and \eqref{eqdhzTrMBQb} yield that
\begin{equation}
P_1=\phi'(P_0')\,\,\text{ and  }\,\,\hal U_1=\phi'(\hal U_0').
\label{eqhdmssztrGt}
\end{equation}
Computing by  Lemma~\ref{lemmaBlowUp} as in  \eqref{eqKrPdJvsT}, we obtain that
\begin{align}
&\hal U_1(\xi)   \eeqref{eqVxhsxzTfHh} \hthet{P_1}\xi (\hal U_1)  \eeqref{eqhdmssztrGt} 
(\hthet{P_1}\xi \circ \phi')(\hal U_0') \cr
& \overset {\textup{Lem.\,\ref{lemmaBlowUp}}} = 
(\phi'\circ\hthet{P_0'}\xi)(\hal U_0') 
\eeqref{eqdhzTrMBQb} 
\phi' \bigl(\hthet{P_0'}\xi(\hal U'_0)\bigr)=\phi'(\hal U'_0(\xi)). 
\label{eqhdzTbnFrWrmSf}
\end{align} 
According to  Figure~\ref{fightmThGdrrgtT},  
the directed line through the ``left leg'' of $\trap(\xi)$ 
will be denoted by $e$; clearly, $E_0,E_1,A_0\in e$. 
Since $\hthet{A_0}\nu(E_0)=E_1$, $\hthet{A_0}\nu(e)=e$, and $\hthet{A_0}\nu$ preserves supporting lines, it follows from \eqref{alignJsTfvSnpSzrTDb} that $\pair {E_1}e$ is a pointed supporting line of   $\,\hal U_0'(\xi)$. Hence,  $\pair {E_1}e$ is a common pointed supporting line of   $\,\hal U_0'(\xi)$
 and $\,\hal U_1(\xi)$. Thus, $\,\hal U'_0(\xi)$ and $\,\hal U_1(\xi)$ are internally tangent.
This fact and the equalities  \eqref{eqVxhsxzTfHh},   \eqref{eqdhzTrMBQb}, \eqref{eqhdmssztrGt}, and 
\eqref{eqhdzTbnFrWrmSf} show that the assumptions of Lemma~\ref{lemmangszrnwnKztlS}, with $\tuple{P_0',\hal U_0',\hal U_0'(\xi), \phi'}$
instead of $\tuple{P_0,\hal U_0,\hal U_0(\xi), \phi}$, hold.
Hence, we obtain from Lemma~\ref{lemmangszrnwnKztlS}  that 
\begin{equation}
\text{$\hal U_1\subseteq\hal U_0'$ \ or \ $\,\hal U_0'(\xi)\subseteq \hal U_1(\xi)$.}
\label{eqvmzzklTsghWzpPr}
\end{equation}

Now, assume the first inclusion in \eqref{eqvmzzklTsghWzpPr}.  Since $0<\nu<1$ in the definition of $\,\hal U_0'$ in \eqref{eqdhzTrMBQa}, we have that $\,\hal U_0'\subseteq \rhullo2 ( \set{A_0} \cup \hal U_0 )$. This together with 
$\,\hal U_1\subseteq \hal U_0'$ yield that  $\,\hal U_1\subseteq \rhullo2 ( \set{A_0} \cup \hal U_0 )$, whence the third line of
\eqref{eqpbxThmlnygrSmzWrB} holds. This contradicts our indirect assumption and so excludes the first  inclusion  in \eqref{eqvmzzklTsghWzpPr}.

So we are left with the second inclusion  given in \eqref{eqvmzzklTsghWzpPr}, that is,  $\,\hal U_0'(\xi) \subseteq \hal U_1(\xi)$, as shown in Figure~\ref{fightmThGdrrgtT}.
Let $f\neq e$ be the other supporting line of $\,\hal U_0(\xi)$ through $A_0$; see Figure~\ref{fightmThGdrrgtT} again. 
It follows from  \eqref{alignJsTfvSnpSzrTDb} that $f$ is a supporting line of $\,\hal U_0'(\xi)$ as well.
The intersection of $e$ and $f$ with the line segment $[A_1,A_2]$ will be denoted by $A_2^\ast$ and 
$A_1^\ast$, respectively; see Figure~\ref{fightmThGdrrgtT}. Since
$\hal U_0(\xi)\loosin {\hsz}$,  the points $A_1^\ast$ and $A_2^\ast$ are strictly between $A_1$ and $A_2$ in the line segment 
$[A_1,A_2]$. Using that $\xi<1$ and $\,\hal U_0'(\xi) \subseteq \hal U_1(\xi)\loosin \hal U_1(1)=\hal U_1\subseteq {\hsz}$, see \eqref{eqpbxZhFtT}, we have that  $\dist {A_0}{\hal U_0'(\xi) }>0$.  Both $e$ and $f$ can be turned continuously 
around $\,\hal U_0'(\xi) $. The precise meaning of this continuity is given in Cz\'edli and Stach\'o~\cite{czgstacho}, where \emph{pointed} supporting lines are ``slide-turned''. However, the edge-freeness of $\,\hal U_0'(\xi)$ allows us to forget about the uniquely defined support point of a pointed supporting line when we refer to \cite{czgstacho}. So, we turn $e$ around $\,\hal U_0'(\xi) $ counterclockwise  sufficiently small to obtain a supporting line $e'$ of $\,\hal U_0'(\xi)$. Similarly, turn $f$ clockwise sufficiently small to obtain a supporting line $f'$ of $\,\hal U_0'(\xi) $. 
The meaning of ``sufficiently small'' here is that 
\begin{enumerate}[\upshape\kern 4pt ---]
\item
the intersection point $G\in e'\cap f'$ belongs to  
\[\inter{\rhullo2(\set{A_0}\cup \hal U_0'(\xi) )\setminus \hal U_0'(\xi)},
\] 
which is possible since  $\dist {A_0}{\hal U_0'(\xi) }>0$, and, in addition,
\item the intersection points $A_1'\in f'\cap [A_1, A_2]$ and $A_2'\in e'\cap [A_1, A_2]$ exist and they are strictly between $A_1$ and $A_1^\ast$ and $A_2^\ast$ and $A_2$, respectively.
\end{enumerate}
By Lemma~\ref{lemmaperspCs}, we have that $\comet {A_0}{\hal U_0(\xi)}\loosin \comet G {\hal U_0'(\xi) }$. 
This fact together with  $\,\hal U_0(\xi)\subseteq \comet {A_0}{\hal U_0(\xi)}$ yield the loose inclusion $ \hal U_0(\xi) \loosin \comet G {\hal U_0'(\xi) }$.
Since we also have that  $\,\hal U_0(\xi)\loosin {\hsz}$ and the inclusion
$\inter{\hal V_1}\cap\inter{\hal V_2}\subseteq \inter {\hal V_1\cap \hal V_2}$ trivially holds for all $\hal V_1,\hal V_2\subseteq \preal$, we obtain that
\begin{align*}
\hal U_0(\xi) &\loosin\comet G {\hal U_0'(\xi) } \cap {\hsz}
\cr 
&=\rhullo2({\hal U_0'(\xi) }\cup\set{A_1', A_2'})
\subseteq \rhullo2(\hal U_1(\xi)\cup\set{A_1, A_2}).
\end{align*}
This loose inclusion and \eqref{eqpbxTlSnHfgjQwpLnt} yield a (small) positive $\delta\in\real$ such that 
$\xi+\delta<1$ and 
\begin{align*}
\hal U_0(\xi+\delta) &=
\hthet{P_0}{(\xi+\delta)/\xi}(\hal U_0) \loosin \rhullo2(\hal U_1(\xi)\cup\set{A_1, A_2})\cr
&\subseteq \rhullo2(\hal U_1(\xi+\delta)\cup\set{A_1, A_2})
\end{align*}
Hence, $\xi+\delta\in H$, contradicting $\xi=\max H$; see \eqref{eqtxtxiisMax}. This contradiction excludes  the second case where $E_1$ belongs to the ``left leg'' of $\trap(\xi)$.

Finally, it is a contradiction that both cases have been excluded. This completes the proof of Lemma~\ref{lemmaEFmain}
\end{proof}

\section{Getting rid of edges}\label{sectiongettingrid}
Recall that \emph{disks} are convex hulls of circles and \emph{circles} are boundaries of disks. 

\begin{lemma}\label{lemmafnntrsCtn} The intersection of finitely many disks of the plane is an edge-free compact convex set whenever it is not empty.
\end{lemma}

\begin{proof} Let $\hal D_1,\dots,\hal D_n$ be disks such that $\,\hal U:= \hal D_1\cap\dots\cap\hal D_n\neq\emptyset$. Clearly, $\,\hal U$ is compact and convex. For the sake of contradiction, suppose that $\,\hal U$ is not edge-free. Then we can pick a supporting line $\ell$ and   $2n+1$ distinct points $P_1,\dots,P_{2n+1}$ belonging to $\ell\cap\partial\hal U$. Since $\inter{\hal D_1}\cap\dots\cap \inter{\hal D_1}\subseteq\inter {\hal U}$, none of the points $P_i$   belongs to this intersection. Hence, for each $i\in\set{1,\dots,2n+1}$, there is a $j=j(i)$ in $\set{1,\dots,n}$ such that $P_i\notin \inter{\hal D_{j(i)}}$. But $P_i\in \hal U\subseteq \hal D_{j(i)}$, whence $P_i\in \ell\cap  \partial\hal D_{j(i)}$. By the pigeonhole principle, there are pairwise distinct $i_1,i_2,i_3\in \set{1,\dots 2n+1}$ such that $j(i_1)=j(i_2)=j(i_3)$. Letting $j$ be this common value, $\set{P_{i_1},P_{i_2},P_{i_3}}\subseteq \ell\cap  \partial\hal D_{j}$. This is a contradiction, because a line and a circle can have at most two points in common.
\end{proof}

For a positive $d\in\real$ and a compact convex subset $\,\hal U$ of $\preal$, we define the \emph{open extension} $\enlarge {\hal U} d$ of $\,\hal U$ by $d$ as 
\begin{align}
\enlarge {\hal U} d := {}&\set{X\in \preal: (\exists Y\in \hal U)\,(\dist X Y  <  d)}\cr
= {}&\set{X\in \preal:  \dist {\set X }{\hal U}  <  d)} ;
\label{eqenlarge}
\end{align}
the second equality above is a consequence of the compactness of $\,\hal U$. Clearly, $\enlarge {\hal U} d$ is an open  set.  For convex compact sets $\,\hal U,\hal V\subseteq \preal$ such that $\,\hal U\subseteq \hal V$, we define the \emph{abundance} of $\hal V$ over $\,\hal U$ as
\begin{equation}
\abund {\hal U}{\hal V} :=\inf\set{d\in \real: 0 < d\text{ and }\hal V\subseteq \enlarge {\hal U} d}. 
\label{eqabund}
\end{equation}
In order to reduce the general case of Theorem~\ref{thmmain} to the edge-free case covered by Lemma~\ref{lemmaEFmain}, we are going to prove the following lemma.

\begin{lemma}\label{lemmagetridedge}
For each nonempty convex compact subset $\,\hal U$ of the plane $\preal$, there exists a sequence $(\hal U_n\colon n\in\nplu)$ of \emph{edge-free} convex compact subsets of $\preal$ such that 
\begin{enumerate}[\quad\upshape (A)]
\item\label{lemmagetridedgea} $\,\hal U\subseteq \hal U_{n+1}\subseteq \hal U_n$ for all $n\in\nplu$, 
\item\label{lemmagetridedgeb} $\lim_{n\to\infty} \abund {\hal U} {\hal U_n} = 0$, and
\item\label{lemmagetridedgec} $\,\hal U = \bigcap_{n\in\nplu} \hal U_n$.
\end{enumerate}
\end{lemma}

\begin{proof}
A disk $\set{\pair x y: (x-a)^2+(y-b)^2\leq r^2}$ will be called  \emph{rational} if $a$, $b$, and $r$ are rational numbers. By basic cardinal arithmetics, there are only countably many rational disks. Hence, there exists a sequence $(\hal D_n: n\in\nplu)$ consisting of all rational disks that include $\,\hal U$ as a subset. For $n\in\nplu$, let $\,\hal U_n:=\hal D_1\cap\dots\cap \hal D_n$; it is an edge-free compact convex set by Lemma~\ref{lemmafnntrsCtn}, and part \eqref{lemmagetridedgea} of the lemma clearly holds. We can assume that $\,\hal U\subseteq\inter{\hal D_1}$, because otherwise we can interchange $\hal D_1$ with a much larger disk of the sequence. In Figure~\ref{fignbhD}, to save space, only two arcs of $\partial\hal D_1$ are given.

\begin{figure}[ht] 
\centerline
{\includegraphics[scale=1.1]{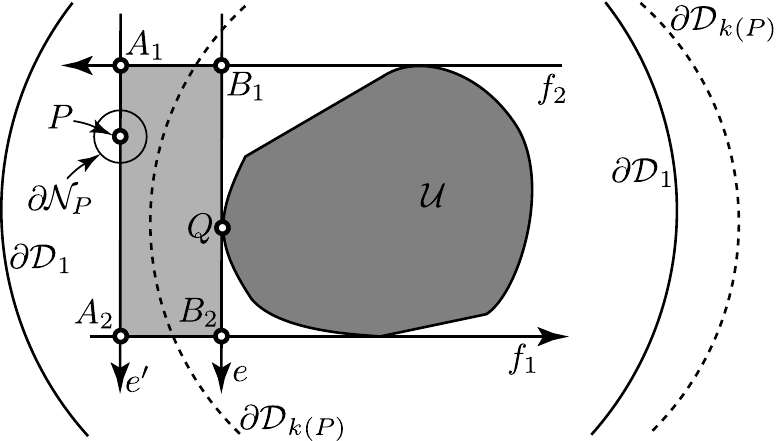}}
\caption{Illustration for the proof of Lemma~\ref{lemmagetridedge}
\label{fignbhD}}
\end{figure}%

For an arbitrary (small) $0<\epsilon\in\real$, let $\hal F(\epsilon):=\hal D_1\setminus \enlarge{\hal U}\epsilon$. Clearly, $\hal F(\epsilon)$ is a compact set. Let $P\in \hal F(\epsilon)$. Since $P\notin \hal U$,  \eqref{eqpbxSpRtlNjbbr} allows us to pick a directed supporting line $e$ of $\,\hal U$ such that $P$ is strictly on the right of $e$; see Figure~\ref{fignbhD}. Let $e'$ be the directed line through $P$ such that $\dir {e'}=\dir e$.
By \eqref{eqpbxZhgTnQspL}, there are exactly two supporting lines,  $f_1$ and $f_2$, that are perpendicular to $e$; their role together with the light-grey rectangle in Figure~\ref{fignbhD} is to show how to choose a rational disk (with sufficiently large radius) containing $\,\hal U$ such that 
each point of this disk is strictly on the left of $e'$. Since this disk belongs to our sequence, it is $\hal D_{k(P)}$ for some $k(P)\in \nplu$; two arcs of  $\partial\hal D_{k(P)}$ are given by dashed curves in the figure. Note that $P\in e'$ need not be between $A_1$ and $A_2$. 
It is clear by the choice of $\hal D_{k(P)}$ that we can pick a circular neighborhood  $\ngbhood P$ of $P$ such that 
$\ngbhood P\cap \hal D_{k(P)}=\emptyset$ and so $\ngbhood P\cap \hal U_{k(P)}=\emptyset$. Note that $\ngbhood P$ is an open set; its boundary is indicated in Figure~\ref{fignbhD}. Since the collection $\set{\ngbhood P: P\in 
\hal F(\epsilon)}$ covers the compact set $\hal F(\epsilon)$, we can select finitely many points $P_1,\dots, P_t$ of $\hal F(\epsilon)$ such that $\hal F(\epsilon)\subseteq \ngbhood {P_1}\cup\dots\cup  \ngbhood {P_t}$. Let $m(\epsilon):=\max\set{k(P_1),\dots,k(P_t)}$. 

Now assume that  $n\geq m(\epsilon)$ and $P\in \hal F(\epsilon)$.
Then $P\in \ngbhood{P_i}$ for some $i\in\set{1,\dots,t}$, whence
$\ngbhood{P_i}\cap \hal U_{k(P_i)}=\emptyset$ yields that $P\notin\hal U_{k(P_i)}$. Thus, for all $n\geq m(\epsilon)$,   
$P\notin\hal U_{n}$ since the  sequence $(\hal U_n: n\in\nplu)$ is decreasing.  This fact together with $\,\hal U_n\subseteq \hal D_1$ imply that, for all $n\geq m(\epsilon)$,  we have that 
$\hal U_n\subseteq \enlarge {\hal U}\epsilon$. This means that 
$ \abund {\hal U} {\hal U_n} \leq \epsilon$ for all $n\geq m(\epsilon)$, proving part \eqref{lemmagetridedgeb} of Lemma~\ref{lemmagetridedge}.

Finally,  part \eqref{lemmagetridedgec} follows from  part \eqref{lemmagetridedgeb}, since every point outside $\,\hal U$ is at a positive distance from $\,\hal U$. This completes the proof of 
of Lemma~\ref{lemmagetridedge}.
\end{proof}

Now, armed with the preparatory lemmas that we have proved so far, we are in the position to prove our theorem.

\begin{proof}[Proof of Theorem~\textup{\ref{thmmain}}]
By \eqref{eqtxtmndgknvx} and Lemma~\ref{lemmasinglton}, we can assume that $\,\hal U_0$ and $\,\hal U_1$ are \emph{convex} compact sets and none of them is a singleton. Assume also that they satisfy condition \ref{thmmain}\eqref{thmmaina} or condition \ref{thmmain}\eqref{thmmainb}, and let $\phi$ denote the transformation from \ref{thmmain}\eqref{thmmaina} or \ref{thmmain}\eqref{thmmainb}, respectively. Then $\,\hal U_1=\phi(\hal U_0)$.  Finally, assume the premise of \eqref{eqpbxThmlnygrSmzWrB}. 
Let $(\hal U_{0,n}: n\in\nplu)$ be a sequence provided for $\,\hal U_0$ by Lemma~\ref{lemmagetridedge} with the notational change that we write $\,\hal U_0$ and $\,\hal U_{0,n}$ instead of $\,\hal U$ and $\,\hal U_n$, respectively. Since $\phi$  preserves the validity of Lemma~\ref{lemmagetridedge}, the statement of this lemma holds also for $\,\hal U_1$ and $\,\hal U_{1,n}:=\phi(\hal U_{0,n})$ instead of $\,\hal U$ and $\,\hal U_n$, respectively. 

Let $W=(A_0+A_1+A_2)/3$ be the barycenter of $\hsz$.
For each $n\in\nplu$ and $i\in\set{0,1,2}$, let $A_i^{(n)}:=\hthet W{(n+1)/n}(A_i)$, and let 
\begin{align*}
\khsz n:=  \hthet W{(n+1)/n}({\hsz}) = \rhullo2{\set{A_0^{(n)},A_1^{(n)},A_2^{(n)}}}.
\end{align*}
Observe that part \eqref{lemmagetridedgeb} of Lemma~\ref{lemmagetridedge} allows us to pick an integer $t(n)\in\nplu$ 
for each $n\in\nplu$ such that $t(n)\geq n$,  $\,\hal U_{0,t(n)}\subseteq \khsz n$, and $\,\hal U_{1,t(n)}\subseteq \khsz n$.  Also, we know from Lemma~\ref{lemmagetridedge} that $\,\hal U_{0,t(n)}$ is edge-free. Thus,  Lemma~\ref{lemmaEFmain} is applicable and it yields a $j(n)\in\set{0,1,2}$ and a $k(n)\in \set{0,1}$ such that the inclusion in the third line of \eqref{eqpbxThmlnygrSmzWrB} holds with self-explanatory notational changes to be exemplified by \eqref{eqbztGrhBmQxC} soon. 
Although $\pair{j(n)}{k(n)}\in \set{0,1,2}\times \set{0,1}$ may depend on $n$, one of the six possible values occurs for infinitely many $n$. Without loss of generality, to avoid complicated notations, we can assume that  $\pair{j(n)}{k(n)}$ does not depend on $n$, because we could work with a subsequence $(n_1,n_2,n_3,\dots)$ instead of $(1,2,3,\dots)$ otherwise.
Also, by changing the notation if necessary, we can assume that $j(n)=0$ and $k(n)=0$. That is, for all $n\in\nplu$, 
\begin{equation}
\hal U_{1,n} \subseteq \rhullo2(\hal U_{0,n}\cup\set{A_1^{(n)},A_2^{(n)}} ).
\label{eqbztGrhBmQxC}
\end{equation}
As mentioned before, Lemma~\ref{lemmagetridedge} and, in particular, its part \eqref{lemmagetridedgec} hold for $\,\hal U_1$ and $(\hal U_{1,n}:n\in\nplu)$. Combining this fact with \eqref{eqbztGrhBmQxC}, 
we obtain that
\begin{equation*}
\hal U_1\subseteq \bigcap_{n\in\nplu} \rhullo2(\hal U_{0,n}\cup\set{A_1^{(n)},A_2^{(n)}} ).
\end{equation*}

Hence, it suffices to show that 
\begin{equation}
\bigcap_{n\in\nplu} \rhullo2(\hal U_{0,n}\cup\set{A_1^{(n)},A_2^{(n)}} )
\subseteq \rhullo2(\hal U_{0}\cup\set{A_1,A_2} ).
\label{eqzhGrnTmrgk} 
\end{equation}
So assume that $P\in\preal$ belongs to the intersection in \eqref{eqzhGrnTmrgk}. So $P$ belongs to each of the sets we intersect in \eqref{eqzhGrnTmrgk}. Hence,
applying Carath\'eodory's well-known theorem, see, for example, 
Schneider~\cite[Theorem 1.1.4]{Schneider1993},  and using the convexity of $\,\hal U_0$,  we can pick a point $X_n\in \hal U_{0,n}$ such that $P$ is of the form
\begin{equation}
P=\lambda_{0,n}\cdot X_n + \lambda_{1,n} \cdot A_1^{(n)} + \lambda_{2,n}\cdot A_2^{(n)},
\label{eqlmBdkadhzksllVt} 
\end{equation}
where $\vec\lambda_n:=\tuple{\lambda_{0,n},\lambda_{1,n},\lambda_{2,n}}\in [0,1]^3$
such that $\lambda_{0,n}+\lambda_{1,n}+\lambda_{2,n}=1$. This condition on $\vec\lambda_n$ means that $\vec\lambda_n$ belongs to the equilateral triangle 
\[E:=\rhullo 3(\set{\tuple{1,0,0}, \tuple{0,1,0}, \tuple{0,0,1}})
\]
in the 3-dimensional space $\real^3$. Hence, as $(\hal U_{0,n}:n\in\nplu)$ is a decreasing sequence  by Lemma~\ref{lemmagetridedge}\eqref{lemmagetridedgea},  the pair $\pair{X_n}{\vec\lambda_n}$ ranges in the Cartesian product
$ U_{0,0} \times E$, which is a compact subset of $\real^2\times E$ since  $E$ is compact and so is  $\,\hal U_{0,0}$ by Lemma~\ref{lemmagetridedge}. 
Therefore, the sequence  $(\pair{X_n}{\vec\lambda_n}: n\in\nplu)$ has a cluster point $\pair X{\vec\lambda}:=\tuple{X,\tuple{\lambda_0,\lambda_1,\lambda_2}}\in \real^2\times E$.
So this sequence has a subsequence converging to   $\pair X{\vec\lambda}$. To simplify the notation again,  
we can assume that this subsequence is the whole sequence; 
without this assumption, the argument is similar but needs more complicated notations. Forming the limit of \eqref{eqlmBdkadhzksllVt} and using that $\lim_{n\to\infty}A_i^{(n)}= A_i$ for $i\in\set{1,2}$, we obtain that 
\begin{equation}
P=\lambda_{0}\cdot X + \lambda_{1} \cdot A_1 + \lambda_{2}\cdot A_2.
\label{eqlmBdkadhmgVhtszhTt} 
\end{equation}
Since $E$ is a compact set, it contains $\vec\lambda=\tuple{\lambda_0,\lambda_1,\lambda_2}$, that is, we have a convex linear combination in \eqref{eqlmBdkadhmgVhtszhTt}. It follows easily from 
Lemma~\ref{lemmagetridedge}\eqref{lemmagetridedgeb} that $X=\lim_{n\to\infty}X_n\in \hal U_0$. These two facts and  \eqref{eqlmBdkadhmgVhtszhTt} imply that $P\in \rhullo 2(\hal U_0\cup\set{A_1,A_2})$. This proves \eqref{eqzhGrnTmrgk} since $P$ was an arbitrary point in the intersection on the left of  \eqref{eqzhGrnTmrgk}. 
The proof of Theorem~\ref{thmmain} is complete.
\end{proof}

\section*{Acknowledgement}
This research of the first author and that of the second author was supported by the Hungarian Research Grants KH 126581 and K~116451, respectively.

\end{document}